\documentclass{amsart}
\pdfoutput=1

\usepackage{amsmath,bm}
\usepackage{amssymb}
\usepackage{microtype}
\usepackage{tikz}
\usetikzlibrary{arrows}
\usepackage{booktabs}
\usepackage[pdftex,colorlinks,citecolor=black,linkcolor=black,urlcolor=black,bookmarks=false]{hyperref}

\usepackage{doi}
\usepackage{subfigure}
\usepackage{algorithm}
\usepackage{algorithmic}

\newtheorem{theorem}{Theorem}[section]
\newtheorem{proposition}[theorem]{Proposition}
\newtheorem{lemma}[theorem]{Lemma}
\newtheorem{corollary}[theorem]{Corollary}

\newtheorem{example}[theorem]{Example}
\newtheorem{remark}[theorem]{Remark}

\theoremstyle{definition}

\numberwithin{figure}{section}
\numberwithin{equation}{section}
\numberwithin{table}{section}

\newcommand{\R}{\mathbb{R}}

\newcommand{\supp}{\operatorname{supp}}

\DeclareMathOperator{\ir}{ir}
\DeclareMathOperator{\clir}{clir}
\DeclareMathOperator{\inte}{int}
\DeclareMathOperator{\conv}{conv}
\DeclareMathOperator{\pos}{pos}
\DeclareMathOperator{\relint}{relint}

\DeclareMathOperator{\argmin}{argmin}
\DeclareMathOperator{\rank}{rank}

\date{\today}

\title[Learning polytopes]{Learning polytopes with fixed facet directions}

\author{Maria Dostert and Katharina Jochemko}

\begin{document}
\address[MD]{Department of Mathematics\\
  KTH Royal Institute of Technology, Stockholm, Sweden}
\email{dostert@kth.se}
\address[KJ]{Department of Mathematics\\
  KTH Royal Institute of Technology, Stockholm, Sweden}
\email{jochemko@kth.se}

\begin{abstract}
We consider the task of reconstructing polytopes with fixed facet directions from finitely many support function evaluations.  We show that for a fixed simplicial normal fan the least-squares estimate is given by a convex quadratic program. We study the geometry of the solution set and give a combinatorial characterization for the uniqueness of the reconstruction in this case.  We provide an algorithm that, under mild assumptions, converges to the unknown input shape as the number of noisy support function evaluations increases. We also discuss limitations of our results if the restriction on the normal fan is removed.

\end{abstract}

\maketitle

\section{Introduction}

The task of reconstructing an unknown geometric object from possibly noisy data such as information about sections and projections or support function evaluations appears naturally in many areas of science and engineering. Application fields include, for example, computerized tomography, computer vision, robotics, and magnetic resonance imaging \cite{RoboticsCole, RoboticsAndComputerVision, ImageProcessing, MinimalData}. Geometric reconstruction problems give rise to interesting mathematical questions, for instance, about the convergence and uniqueness of the reconstruction. The study of these questions oftentimes require methods from a variety of fields such as combinatorics, convex and computational geometry, functional analysis, optimization and statistics. 

In the present article we study the problem of learning a polytope from possibly noisy evaluations of its support function. This question arises, for example, in target reconstruction from laser data~\cite{LaserData}, cone-beam based computerized tomography~\cite{Gregor2002FastFR}, projection magnetic resonance imaging~\cite{MRI}, and robotics where the data arises from grasps by a robot gripper~\cite{PrinceWillsky, Schneiter}.  See also~\cite{GardnerNew} and~\cite[pp. 135-136]{GeometricTomography} as well as references therein.

Given a convex body $P$ in $\mathbb{R}^d$, the support function $h_P \colon \mathbb{R}^d\rightarrow \mathbb{R}$  of $P$ is defined by
\[
h_P(\mathbf{u})=\max _{\mathbf{x}\in P} \langle \mathbf{x},\mathbf{u}\rangle
\]
for all $\mathbf{u}\in \mathbb{R}^d$. Every convex body is uniquely determined by its support function. The challenge is to retrieve the convex body from incomplete or noisy information about the support function. Given a data set
\[
\left\{(\mathbf{u}^{(i)},y^{(i)}) \colon y^{(i)}=h_P(\mathbf{u}^{(i)}) + \varepsilon ^{(i)} \right\}_{i=1}^m \subset \mathbb{R}^d\times \mathbb{R}
\]
consisting of pairs of directions in $\mathbb{R}^d$ and corresponding (possibly noisy) support function evaluations, a general approach is to estimate the unknown input shape by using the least-squares method. For a given set $\mathcal{K}$ of convex bodies in $\mathbb{R}^d$ the least-squares estimator is defined as
\[
\hat{P}(U,\mathbf{y}) = \argmin \nolimits_{P\in \mathcal{K}} \frac{1}{m}\sum _{i=1}^m \left(h_P (\mathbf{u}^{(i)})-y^{(i)}\right)^2 \, .
\]
If $\mathcal{K}$ is the set of all $0$-dimensional convex bodies in $\mathbb{R}^d$ then the support function is given by the scalar product. In this setting, the estimator is the solution obtained by ordinary linear regression. The support function of any convex body is positively homogeneous and convex, and these properties also characterize support functions of convex bodies. Therefore, in general, the task of reconstruction a convex body from support function evaluations is equivalent to fitting a positively homogeneous and convex function to given data points. 

The estimator and its computational tractability as well as its uniqueness depends on the choice of $\mathcal{K}$. Foundational works by Prince and Willsky~\cite{PrinceWillsky}, Lele, Kulkarni and Willsky~\cite{LaserData}, Gardner and Kinderlen~\cite{GardnerNew} and Gardner, Kinderlen and Milanfar~\cite{gardner2006convergence} investigate reconstruction algorithms for the case that $\mathcal{K}$ consists of all convex bodies. In ~\cite{Guntuboyina} Guntuboyina considers this problem for polytopes with increasing maximal number of vertices and shows convergence rates. In recent work, Soh and Chandrasekaran~\cite{Soh} argue that even though the Prince-Willsky algorithm~\cite{PrinceWillsky} converges, as was shown in~\cite{gardner2006convergence}, the least-squares estimator provides in general only little information about the facial structure, depending on the underlying convex body. They address this by considering the case when $\mathcal{K}$ is a finitely parametrized set of convex bodies consisting of linear images of specific sets, in this case, simplices and free spectrahedra. In this article we add to this body of work on parametric reconstruction by investigating the question of reconstructing polytopes with given facet directions parametrized by their facet displacements.
\subsubsection*{Main contributions} In the following we summarize the main contributions of this article. For precise definitions of the terms used see Section~\ref{sec:prelim}.

We consider the family of polytopes with given fixed facet directions $\mathbf{v}_1,\ldots, \mathbf{v}_n \in \mathbb{R}^d$. Every non-empty polytope
\[
P=\{\mathbf{x}\in \mathbb{R}^d \colon \langle \mathbf{v}_i , \mathbf{x} \rangle \leq b_i , i=1,\ldots , n \} \, , \quad b_1,\ldots, b_n \in \mathbb{R},
\]
is uniquely determined by its facet displacements $h_P(\mathbf{v}_i)=h_i \leq b_i$, $i=1,\ldots, n$, and we write $P=P(\mathbf{h})$ where $\mathbf{h}=(h_1,\ldots, h_n)$ is the vector of displacements called support vector. Note, that if $b_i$ is greater than $h_i$ then the inequality $\langle\mathbf{v}_i,\mathbf{x}\rangle \leq b_i$ is redundant. The set of all such polytopes can be partitioned according to their normal fans. In particular, if two polytopes belong to the same partition then they have the same combinatorial type. The set of all polytopes with a given normal fan $\Delta$ has the structure of an open polyhedral cone, called type cone~\cite{McMullen}. The closure of this type cone, denoted $\mathcal{P}(\Delta)$ consists of all deformations of polytopes with normal fan $\Delta$. Geometrically, a polytope is a deformation of a polytope if it can be obtained by movements of the facets, while keeping their directions, without passing a vertex. Type cones and deformations are a classical topic in discrete geometry that played a prominent role in McMullen's seminal work on the polytope algebra~\cite{PolytopeSimple,PolytopeAlgebra}.

We focus on the case when $\Delta$ is simplicial as a fan. The restriction to simplicial fans can be justified since every polytope is a deformation of a simple polytope (with a simplicial normal fan)~\cite{Shephard}. 
Given a simplicial polytopal fan $\Delta$ and input-output data $\{(\mathbf{u}^{(i)},y^{(i)})\}_{i=1}^m\subset \mathbb{R}^d \times \R$ we show (Theorem~\ref{thm:main}) that the least-squares estimate is given by
\[
\hat{P}^{\Delta}(U,\mathbf{y}) = \argmin _{P(\mathbf{h})\in \mathcal{P}(\Delta)}\| A_U\mathbf{h}-\mathbf{y}\| \, ,
\]
where $A_U \in \mathbb{R}^{m\times n}$ is a sparse matrix that only depends on $\Delta$ and the directions $U=(\mathbf{u}^{(1)},\ldots, \mathbf{u}^{(m)})^\mathsf{T}$. Hence, for a given simplicial fan $\Delta$ and $\mathcal{K}=\mathcal{P}(\Delta)$ we obtain that the least-squares estimator $\hat{P}^\Delta(U,\mathbf{y})$ is the solution of a convex quadratic program and therefore globally optimal. Moreover, the set of all type cones of polytopal fans with fixed ray generators $\mathbf{v}_1,\ldots, \mathbf{v}_n$ forms a polyhedral fan. It follows that if $\mathcal{K}$ is the set of all polytopes with fixed facet directions then the least-squares estimate is the solution of a piecewise quadratic program (Corollary~\ref{cor:piecewisequadratic}). 

Theorem~\ref{thm:main} then allows us to study the geometry of the solution set $\hat{P}^\Delta(U,\mathbf{y})$. In particular, we see that the solution set $\hat{P}^\Delta(U,\mathbf{y})$ is always a polyhedron in the parameter space. We also study the uniqueness of the reconstruction. In particular, we show in Proposition~\ref{prop:injective} that the set of all matrices $U$, for which $\hat{P}^\Delta(U,\mathbf{y})$ is unique, is a semi-algebraic set. Further, we provide a combinatorial characterization for the uniqueness of the reconstruction. We construct a bipartite graph $G_U$ with vertex set $[n]\sqcup [m]$ that only depends on $\Delta$ and the matrix of directions $U$ such that $\hat{P}^\Delta(U,\mathbf{y})$ is almost surely unique for all $\mathbf{y}$ if and only if $G_U$ has a matching of cardinality $n$. In particular, a unique reconstruction requires at least as many data than facet directions.

We also consider the convergence of the least-squares estimator for sequences of data consisting of noisy support function evaluations. Given input data $\mathbf{u}^{(1)},\mathbf{u}^{(2)},\ldots,$ that are sufficiently concentrated around every single ray generator $\mathbf{v}_i$ of $\Delta$ and an unknown input shape $P\in \mathcal{P}(\Delta)$.  Under these assumptions, we show in Theorem \ref{thm:convergence} that the least-squares estimator converges almost surely to $P$ in Hausdorff distance under mild assumptions on the noise $\{\varepsilon^{(i)}\}_{i\geq 1}$. In particular, if the noise is normally distributed with bounded variance the convergence rate is $O(1/\sqrt{m})$ up to a fixed failure probability. 

If we weaken our assumptions from $\mathcal{K}$ being polytopes in $\mathcal{P}(\Delta)$ to arbitrary polytopes with facet directions $\mathbf{v}_1,\ldots, \mathbf{v}_n$ then the uniqueness and convergence results do not hold in general. We provide examples illustrating the limitations.

\begin{figure}[!h]
\centering
\begin{picture}(100,90)
\put(-100,0){\includegraphics[width=0.26\textwidth]{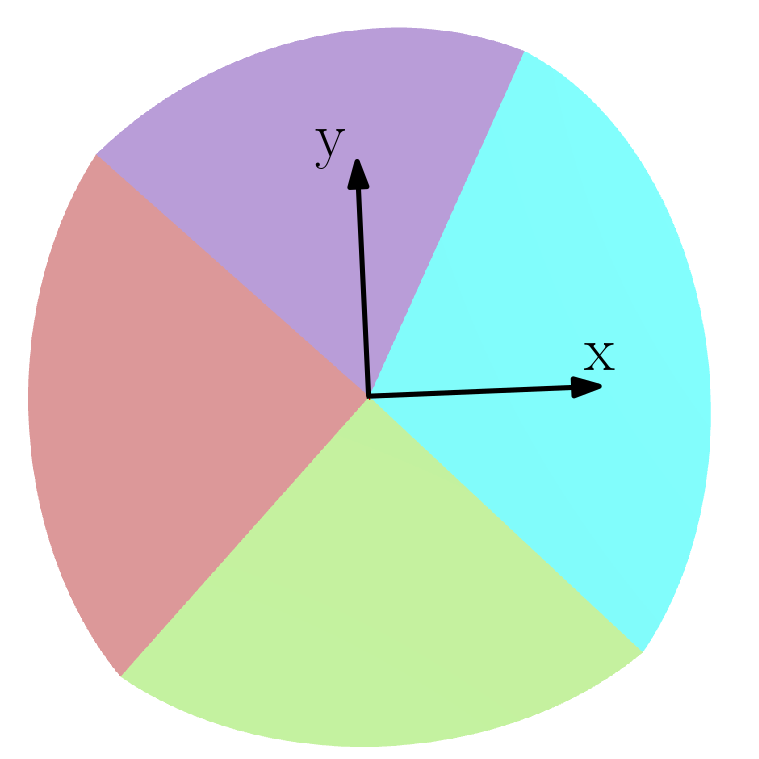}}
\put(40,0){\includegraphics[width=0.49\textwidth]{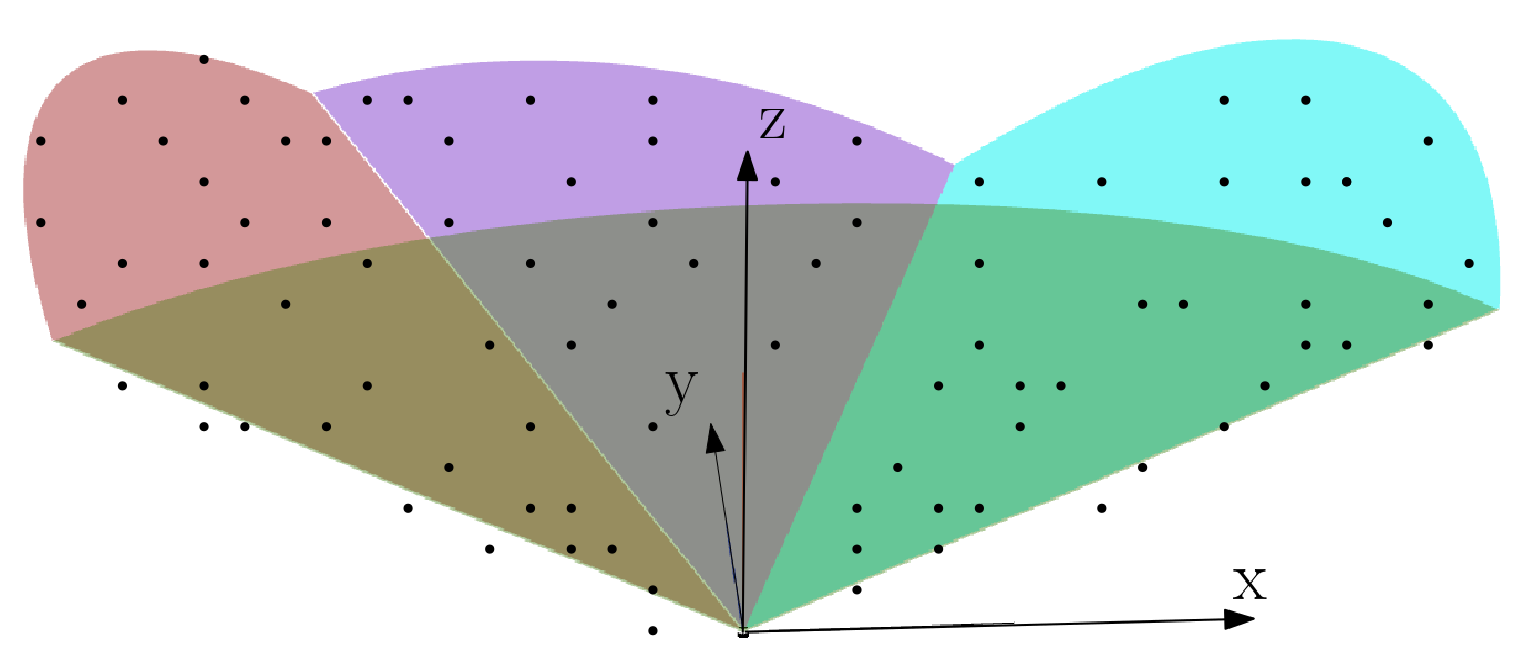}}
\end{picture}
\caption{Left: Conical regions of linearity. Right: Convex piecewise linear function and sample points}
\label{Figure:Application}
\end{figure}

\subsubsection*{Related work}Reconstructing polytopes from support function evaluations may be viewed as fitting a convex, positively homogeneous, piecewise linear function to the given data $\{(\mathbf{u}^{(i)},y^{(i)})\}_{i=1}^m$. Moreover, considering the reconstruction of polytopes with a given simplicial normal fan corresponds to fixing the regions of linearity and is therefore an instance of piecewise linear regression. See Figure~\ref{Figure:Application} for an illustration. This fits into the general framework of max-affine regression which has also been studied in~\cite{Balzs2016ConvexRT,Hannah2013,Magnani2009,Soh}. More recently, in~\cite{Reilly} O'Reilly and Chandrasekaran generalized this further to spectrahedral regression, that is, fitting a spectrahedral function to the data. In contrast to these previous works our setup yields a global optimum. In a slightly different setup, Brysiewicz~\cite{MR4133662} considers the reconstruction of Newton polytopes of hypersurfaces represented numerically via so-called witness sets.

Gardner et al ~\cite{gardner2006convergence} were the first to prove convergence of the Prince-Willsky algorithm~\cite{PrinceWillsky} under the assumption that the directions $\mathbf{u}^{(i)}$ are \textit{evenly spread}, which roughly means that the number of elements of the sequence of directions grows linearly in every neighborhood of any vector in $\mathbb{S}^{d-1}$. In Theorem~\ref{thm:convergence} we prove that our algorithm converges under a similar assumption on the neighborhood of every facet direction $\mathbf{v}_{i}$. Therefore our proof of convergence holds under weaker assumptions.

\subsubsection*{Outline of the paper}
In Section~\ref{sec:prelim} we introduce necessary preliminaries and notations. In Section~\ref{sec:structuralresults} we study the geometry of the solution set. In particular we show in Section~\ref{subsec:solutionset} that for fixed normal fan the least-squares estimator is given by a quadratic program. Furthermore, in Section~\ref{subsec:uniqueness} we study conditions under which the least-squares estimator is unique. Thereafter, we consider the complexity of the computation of the least-squares estimator in Section~\ref{subsec:algorithm}. We also compare our algorithm to the algorithm provided by Gardner and Kinderlen~\cite{GardnerNew} adapted to our setting. In Section~\ref{sec:convergence} we prove convergence of the reconstruction under mild assumptions. Finally, in Section~\ref{sec:ReconstrFixedFacetNormals} we remove the assumption on a normal fan. We give examples which show that our results in general do not hold without the assumption of a fixed normal fan.

\section{Notation and preliminaries}\label{sec:prelim}
\subsection{Cones and fans}
In this section we collect necessary preliminaries from polyhedral geometry and define notation. For further reading we recommend \cite{Gruber,Ziegler}.

In the following we work in the Euclidean space $\mathbb{R}^d$ with standard inner product $\langle .,. \rangle$ and norm $\|.\|$. A \textbf{polyhedron} is defined as the intersection of finitely many halfspaces. A collection $\Delta$ of polyhedra is a \textbf{polyhedral complex} if the following conditions are satisfied:
\begin{itemize}
\item[(i)] if $C\in \Delta$ then also every face of $C$ is contained in $\Delta$, and
\item[(ii)] if $C_1, C_2 \in \Delta$ then $C_1\cap C_2$ is a face of $C_1$ and $C_2$.
\end{itemize}
A \textbf{polyhedral cone} is a finite intersection of closed linear half-spaces in $\mathbb{R}^d$. Equivalently, a set $C$ is a polyhedral cone if there are finitely many vectors $\mathbf{v}_1,\ldots, \mathbf{v}_k\in \mathbb{R}^d$ such that
\[
C \ = \ \pos ( \{\mathbf{v}_1,\ldots, \mathbf{v} _k\}) \ = \ \left\{\sum _{i=1}^k \lambda _i \mathbf{v}_i \colon \lambda _1,\ldots, \lambda _k \geq 0 \right\} \, .
\]
The vectors $\mathbf{v}_1,\ldots, \mathbf{v}_k$ are said to \textbf{generate} the cone $C$. If $C$ is generated by linearly independent vectors then $C$ is called a \textbf{simplicial} cone.

A \textbf{fan} $\Delta$ is a polyhedral complex consisting of cones.
The \textbf{support} of $\Delta$ is the geometric union of all cones in $\Delta$, that is,
\[
\supp (\Delta) = \bigcup _{\sigma \in \Delta} \sigma \subseteq \mathbb{R}^d\, .
\]
The carrier of a vector $\mathbf{x}\in \supp (\Delta)$, denoted $\sigma (\mathbf{x})$, is the unique cone $\sigma\in \Delta$ such that $\mathbf{x}$ is contained in $\relint \sigma$, the relative interior of $\sigma$. The set of all cones of dimension $i$ in $\Delta$ is denoted by $\Delta^{(i)}$.

A fan $\Delta$ \textbf{refines} a fan $\Delta'$ (equivalently, $\Delta'$ \textbf{coarsens} $\Delta$) if every cone in $\Delta '$ is a union of cones in $\Delta$.

Given two fans $\Delta$ and $\Delta'$ with the same support, the \textbf{coarsest common refinement}, denoted $\Delta _1\wedge \Delta _2$, is the polyhedral fan
\[
\Delta _1\wedge \Delta _2 = \{ \sigma _1\cap \sigma _2 \colon \sigma _1 \in \Delta, \sigma _2 \in \Delta'\} \, .
\]

For every polyhedron $P$ in $\mathbb{R}^d$ and every non-empty face $F$ of $P$ the \textbf{normal cone} $N _F (P)$ of $P$ at $F$ is defined as
\[
N_F (P) \ = \ \left\{\mathbf{x}\in \R ^d \colon \max _{\mathbf{y}\in P}\langle \mathbf{x},\mathbf{y}\rangle= \langle \mathbf{x},\mathbf{p}\rangle  \text{ for all }\mathbf{p}\in F\right\} \, .
\]
The collection of all normal cones of $P$ forms the \textbf{normal fan} of $P$,
\[
\mathcal{N} (P) \ = \ \left\{N_F (P) \colon F \text{ non-empty face of } P \right\}\, .
\]
See the hexagonal fan in Figure~\ref{fig:hexagon} for an example. A bounded polyhedron $P\subset \R ^d$ is called a \textbf{polytope}. A polyhedron is a polytope if and only if $\supp \mathcal{N} (P)$ is equal to $\mathbb{R}^d$. If $P$ is full-dimensional this is the case if and only if the ray generators of the normal fan positively span $\mathbb{R}^d$. A fan is called \textbf{polytopal} if it is the normal fan of a polyhedron. Furthermore, a simplicial polytopal fan is a polytopal fan which is simplicial as a fan. Two polyhedra $P$ and $Q$ are called \textbf{normally equivalent} if $\mathcal{N} (P) = \mathcal{N}(Q)$.

The \textbf{support function} $h_P \colon \mathbb{R}^d \rightarrow \mathbb{R}$ of a polytope $P$ is the vector sum
\[
h_P (\mathbf{u})=\max _{\mathbf{y}\in P}\langle \mathbf{u},\mathbf{y}\rangle \, .
\]
The support function is positively homogeneous and convex: for all $\lambda >0$ and all $\mathbf{u}\in \mathbb{R}^d$ it holds that $h_P (\lambda \mathbf{u})=\lambda h_P(\mathbf{u})$ and for all $0\leq \lambda \leq 1$ and all $\mathbf{u},\mathbf{v}\in \mathbb{R}^d$
\[
h_P (\lambda \mathbf{u}+(1-\lambda)\mathbf{v}) \leq \lambda h_P(\mathbf{u})+(1-\lambda)h_P(\mathbf{v}) \, .
\]
These two properties also characterize convex and bounded sets: for every positively homogeneous and convex function there is a unique convex body whose support function coincides with this function. In particular, every polytope is uniquely determined by its support function. The following lemma is immediate.
\begin{lemma}\label{lem:hPlinear}
For every polytope $P$, every non-empty face $F$ of $P$ and every $\mathbf{u}\in N_F(P)$
\[
h_P(\mathbf{u})=\langle \mathbf{u},\mathbf{p} \rangle \quad \text{for all }  \mathbf{p} \in F \, .
\]
In particular, the support function restricted to any normal cone is linear.
\end{lemma}
The \textbf{Minkowski sum} of two polytopes $P$ and $Q$ is defined by
\[
P+Q \ = \ \{p+q \colon p\in P, q\in Q\}
\]
and is again a polytope. The normal fan of the Minkowski sum $P+Q$ is the coarsest common refinement of the normal fans of its summands, that is,
\[
\mathcal{N}(P+Q) = \mathcal{N}(P)\wedge \mathcal{N}(Q)  \, .
\]

The support function is additive under taking Minkowski sums and linear with respect to dilation with positive factors: For all polytopes $P,Q$ and any $\lambda >0$
\[
h_{P+Q}=h_P+h_Q \, ,
\]
and
\[
h_{\lambda P} =\lambda h_{P} \, .
\]
The \textbf{Hausdorff distance} $d_H(P,Q)$ between two subsets $P,Q$ of $\mathbb{R}^d$ is defined as 
\[
d_H(P,Q)= \max \{ \sup \nolimits_{\mathbf{p}\in P}d(\mathbf{p},Q),\sup \nolimits_{\mathbf{q}\in Q}d(\mathbf{q},P)\}
\]
where $d(\mathbf{p},Q)$ denotes the distance of $\mathbf{p}$ to the set $Q$. If $P$ and $Q$ are convex, then $d_H (P,Q)$ can be expressed in terms of their support functions as
\[
d_H(P,Q)=\max _{\mathbf{u}\in \mathbb{S}^d} |h_P (\mathbf{u})-h_Q(\mathbf{u})| \, .
\]
\subsection{The deformation cone $\mathcal{P}(\Delta)$} In this subsection we collect results on deformations and type cones. There are different ways of describing these objects. For our purpose we consider their parametrization via facet displacements as considered in~\cite{McMullen}. For equivalent descriptions see also~\cite{Meyer,Postnikov,PRW,Toricvarieties}. Let $V=\{\mathbf{v}_1,\ldots,\mathbf{v}_n\}\subset \mathbb{R}^d$ be positively spanning vectors. Then for all $\mathbf{h}\in \mathbb{R}^n$
\[
P \ = \ \{\mathbf{x}\in \mathbb{R}^d \colon \langle \mathbf{x},\mathbf{v_i}\rangle \leq h_i\text{ for all }1\leq i\leq n\}
\]
is a polytope. Vectors $\mathbf{h}\in \mathbb{R}^n$ for which this polytope is non-empty are called \textbf{compatible}. A compatible vector $\mathbf{h}$ is furthermore called \textbf{irredundant} if removing any inequality $\langle \mathbf{x},\mathbf{v_i}\rangle \leq h_i$ changes the polytope. The set of irredundant vectors $\mathbf{h}\in \mathbb{R}^d$, denoted $\ir (V)$ has the structure of an open polyhedral cone. For every vector $\mathbf{h}$ in the closure of $\ir (V)$, which is denoted by $\clir (V)$, it holds that $h_i=h_P(\mathbf{v}_i)$. It follows that there is a one-to-one correspondence between vectors in $\clir (V)$ and polytopes with facet directions $\mathbf{v}_1,\ldots, \mathbf{v}_n$. For all $\mathbf{h}\in \clir (V)$ we write $P(\mathbf{h})=P$. We will oftentimes identify $P(\mathbf{h})$ with its \textbf{support vector} $\mathbf{h}=(h_{P(\mathbf{h})}(\mathbf{v}_1),\ldots, h_{P(\mathbf{h})}(\mathbf{v}_n))$.

Let $\Delta$ be a polytopal fan with ray generators $V=\{\mathbf{v}_1,\ldots,\mathbf{v}_n\}\subset \mathbb{R}^d$. The set of all polytopes with normal fan $\Delta$ is closed under taking Minkowski sums and dilations by positive reals. It has thus the structure of an abstract open cone and is called \textbf{type cone} of $\Delta$, denoted $\mathcal{T}(\Delta)$. The closure of the type cone consists of all polytopes whose normal fan is a coarsening of $\Delta$. This closed cone, denoted $\mathcal{P}(\Delta)$, is called the \textbf{deformation cone} of $\Delta$. Equivalently, if $P$ is a polytope with normal fan $\Delta$ then the cone $\mathcal{P}(\Delta)$ consists of all polytopes $Q$ such that $\lambda Q+R=P$ for some polytope $R$ and some $\lambda >0$. Let $\mathbf{h},\mathbf{h}'$ be support vectors of two polytopes $P(\mathbf{h}), P(\mathbf{h}') \in \mathcal{P}(\Delta)$. Then
\[
P(\mathbf{h})+P(\mathbf{h}')=P(\mathbf{h}+\mathbf{h}') \quad \text{ and } P(\lambda \mathbf{h})=\lambda P(\mathbf{h}) \quad \text{for } \lambda >0 \, .
\]
Therefore, the map $\mathcal{P}(\Delta) \rightarrow \R^n, P(\mathbf{h}) \mapsto \mathbf{h}$ defines a linear embedding of $\mathcal{P}(\Delta)$ into $\mathbb{R}^n$. If $\Delta$ is a simplicial polytopal fan, then for all $\mathbf{h} \in \mathcal{T}(\Delta)$, the polytope $P(\mathbf{h})$  is a simple polytope. Since the property of being simple is stable under small pertubations of $\mathbf{h}$ the cone $\mathcal{P}(\Delta)$ is of dimension $n$. Moreover, the set of deformation cones form a polyhedral subdivision of $\clir (V)$~\cite{McMullen}, in particular, $\clir (V)$ is the union of $\mathcal{P}(\Delta)$ over all simplicial polytopal fans $\Delta$ with ray generators $V=\{\mathbf{v}_1,\ldots,\mathbf{v}_n\}$. It follows that every polytope with facet directions $V$ is a deformation of a simple polytope.

If $\Delta$ is a simplicial polytopal fan then the deformation cone is given by the following \textbf{wall-crossing inequalities} as given in~\cite{wallcrossing}. See also~\cite{Ivan}.
\begin{proposition}[{\cite[Lemma 2.1]{wallcrossing}}]\label{prop:wallcrossing}
Let $\Delta$ be a simplicial polytopal fan in $\mathbb{R}^d$ with ray generators $V=\{\mathbf{v}_1,\ldots,\mathbf{v}_n\}$. Then $\mathcal{P}(\Delta)\subset \mathbb{R}^n$ is given by all inequalities of the form
\[
\sum _{i=1}^{d+1}c_{j_i}h_{j_i}\geq 0
\]
where $\mathbf{v}_{j_1},\ldots, \mathbf{v}_{j_{d+1}}$ are the generators of neighboring maximal cells $\rho _1,\rho _2$ in $\Delta$ such that $\mathbf{v}_{j_2}\not \in \rho _1$ and $\mathbf{v}_{j_1}\not \in \rho _2$, and 
\[
\sum _{i=1}^{d+1}c_{j_i}\mathbf{v}_{j_i}= 0
\]
is the unique linear dependence relation such that $c_{j_1}+c_{j_2}=2$.
\end{proposition}

\begin{example}\label{ex:1}
Let $\Delta$ be the complete simplicial fan in $\mathbb{R}^2$ with rays $\sigma _k$ spanned by the unit vectors $\mathbf{v}_k=e^{\frac{2\pi i (k-1)}{6}}$ and maximal cones $\rho _k=\pos \{\mathbf{v}_k,\mathbf{v}_{k+1}\}$ for all $k\in \mathbb{Z}/6\mathbb{Z}$. (See Figure~\ref{fig:hexagon}.) The fan $\Delta$ equals the normal fan of the hexagon and is thus polytopal.  By considering each ray and its two neighboring maximal cells we obtain the linear dependence relation
\[
\mathbf{v}_i+\mathbf{v}_{i+2}-\mathbf{v}_{i+1}=\mathbf{0}
\]
for all $k\in \mathbb{Z}/6\mathbb{Z}$. By Proposition~\ref{prop:wallcrossing}, the cone $\mathcal{P}(\Delta)$ is thus equal to
\[
\mathcal{P}(\Delta) \ = \ \{ \mathbf{h}\in \mathbb{R}^6 \colon h_{k} + h_{k+2} - h_{k+1} \geq 0 \} \, .
\]
\end{example}
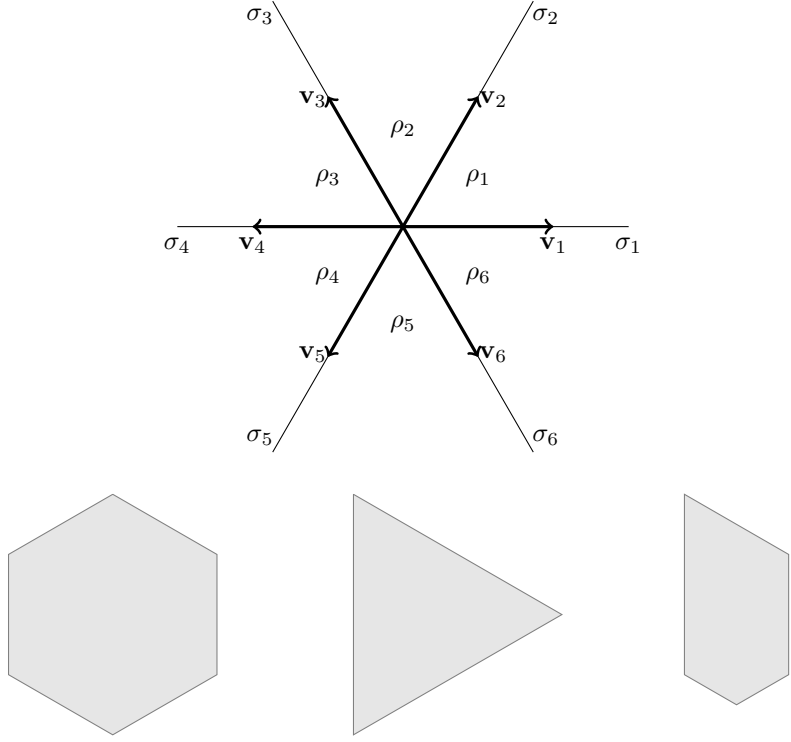
\begin{figure}
\begin{tikzpicture}
\draw[-] (-3,0) -- (3,0);
\draw[-] (-1.732,-3) -- (1.732,3);
\draw[-] (-1.732,3) -- (1.732,-3);
\draw[<->,very thick] (-2,0) -- (2,0);
\draw[<->, very thick] (-1,-1.732) -- (1,1.732);
\draw[<->, very thick] (-1,1.732) -- (1,-1.732);
\node at (2, -0.25) {$\mathbf{v}_1$}; 
\node at (-2, -0.25) {$\mathbf{v}_4$};
\node at (1.2, 1.7) {$\mathbf{v}_2$}; 
\node at (-1.2, -1.7) {$\mathbf{v}_5$};
\node at (-1.2, 1.7) {$\mathbf{v}_3$}; 
\node at (1.2, -1.7) {$\mathbf{v}_6$};

\node at (3, -0.25) {$\sigma_1$}; 
\node at (-3, -0.25) {$\sigma_4$};
\node at (1.9, 2.8) {$\sigma_2$}; 
\node at (-1.9, -2.8) {$\sigma_5$};
\node at (-1.9, 2.8) {$\sigma_3$}; 
\node at (1.9, -2.8) {$\sigma_6$};

\node at (1, 0.65) {$\rho_1$}; 
\node at (1, -0.65) {$\rho_6$}; 
\node at (-1, 0.65) {$\rho_3$}; 
\node at (-1, -0.65) {$\rho_4$}; 
\node at (0, 1.3) {$\rho_2$}; 
\node at (0, -1.3) {$\rho_5$}; 

\end{tikzpicture}
\\
\vspace{0.5cm}
\begin{tikzpicture}[scale=0.8]
\draw[color={gray}, fill=gray!20] (-8,2) -- (-6.268,1)--(-6.268,-1)--(-8,-2)--(-9.732,-1)--(-9.732,1)--(-8,2);

\draw[color={gray}, fill=gray!20] (-4,2)--(-4,-2)--(-0.536,0)--(-4,2);

\draw[color={gray}, fill=gray!20] (1.5,-1)--(1.5,2)--(3.232,1)--(3.232,-1)--(2.366,-1.5)--(1.5,-1);
\end{tikzpicture}
\caption{The hexagonal fan $\Delta$ of Example~\ref{ex:1} and examples of polytopes in $\mathcal{P}(\Delta)$.}
\label{fig:hexagon}
\end{figure}

\begin{example}\label{ex:3d}
We consider the set of facet directions $V = \{\mathbf{v}_1, \ldots, \mathbf{v}_5\}$ where 
\begin{align}\label{example:vi}
\mathbf{v}_1 = \begin{pmatrix} 0 \\ 1 \\1\end{pmatrix},
\mathbf{v}_2 = \begin{pmatrix} 0 \\ -1 \\1\end{pmatrix},
\mathbf{v}_3 = \begin{pmatrix} 1 \\ 0 \\1\end{pmatrix},
\mathbf{v}_4 = \begin{pmatrix} -1 \\ 0 \\1\end{pmatrix},
\mathbf{v}_5 = \begin{pmatrix} 0 \\ 0 \\-1\end{pmatrix}.
\end{align}\\
For $\mathbf{h}_1=(4,4,2,2,0)^\mathsf{T}$, $\mathbf{h}_2=(2,2,4,4,0)^\mathsf{T}$ and $\mathbf{h}_3=(2,2,2,2,0)^\mathsf{T}$, let $P_1 = P(\mathbf{h}_1)$, $P_2 = P(\mathbf{h}_2)$ and $Q = P(\mathbf{h}_3)$ be the polytopes depicted in Figure~\ref{ExampleP1dim3}. 

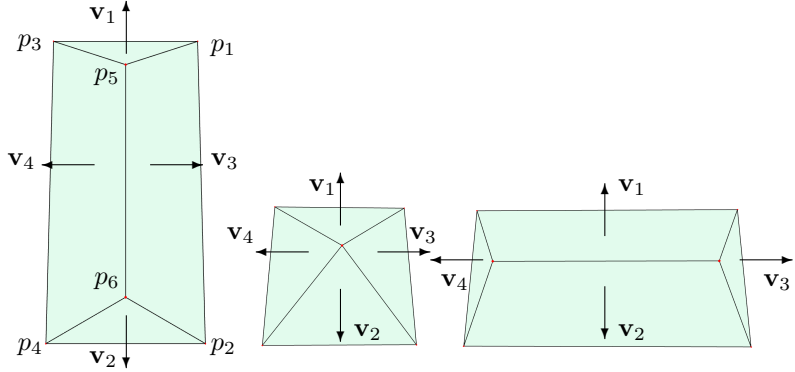
\begin{figure}[!h]
\centering
\begin{picture}(200,135)
\put(-40,0){\includegraphics[width=0.2\textwidth]{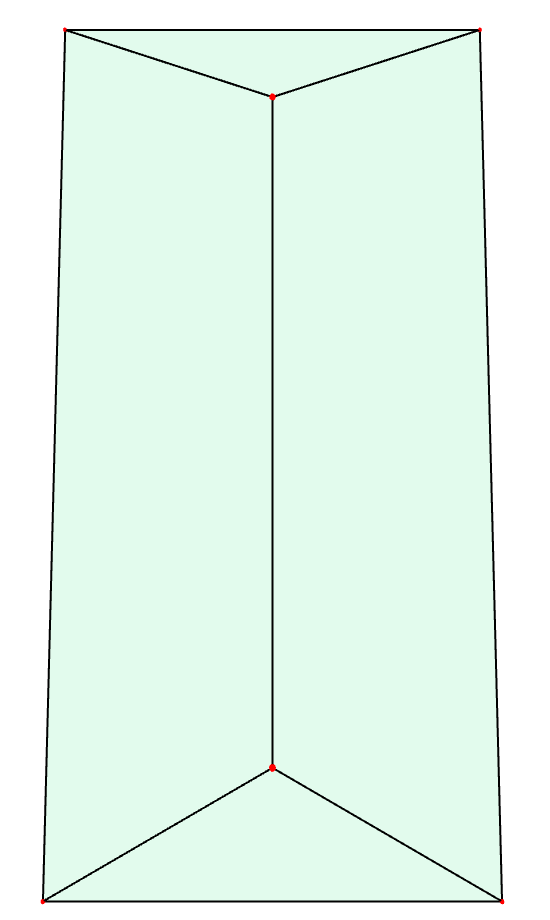}}
\put(45,0){\includegraphics[width=0.175\textwidth]{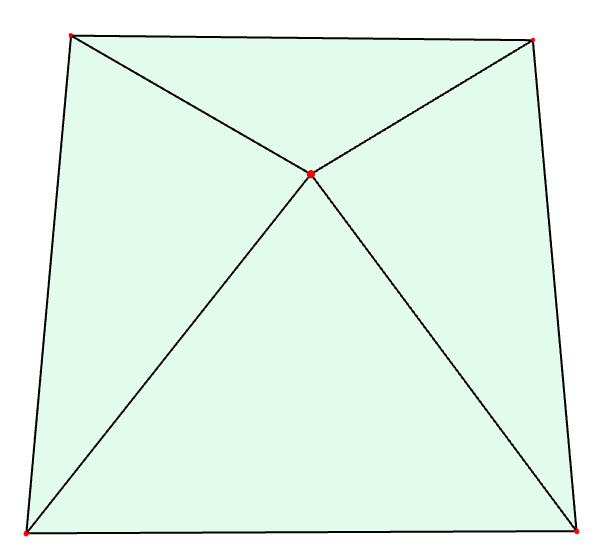}}
\put(120,-2){\includegraphics[width=0.32\textwidth]{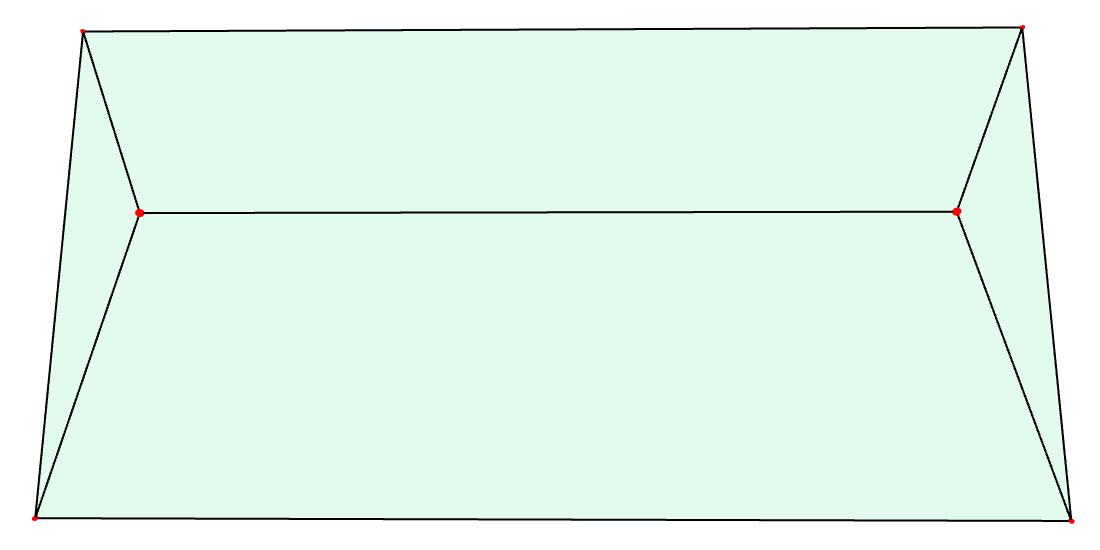}}
\put(-45,0){$p_4$}
\put(28,0){$p_2$}
\put(-45,115){$p_3$}
\put(28,113){$p_1$}
\put(-16,102){$p_5$}
\put(-16,25){$p_6$}
\put(-4,13){\vector(0,-1){20}}
\put(-4,112){\vector(0,1){20}}
\put(5,70){\vector(1,0){20}}
\put(-16,70){\vector(-1,0){20}}
\put(-18,125){$\mathbf{v}_1$}
\put(-18,-5){$\mathbf{v}_2$}
\put(28,70){$\mathbf{v}_3$}
\put(-49,70){$\mathbf{v}_4$}
\put(77,23){\vector(0,-1){20}}
\put(77,47){\vector(0,1){20}}
\put(91,37){\vector(1,0){20}}
\put(65,37){\vector(-1,0){20}}
\put(65,60){$\mathbf{v}_1$}
\put(82,5){$\mathbf{v}_2$}
\put(103,42){$\mathbf{v}_3$}
\put(34,42){$\mathbf{v}_4$}
\put(177,24){\vector(0,-1){20}}
\put(177,43){\vector(0,1){20}}
\put(228,34){\vector(1,0){20}}
\put(131,34){\vector(-1,0){20}}
\put(182,60){$\mathbf{v}_1$}
\put(182,5){$\mathbf{v}_2$}
\put(237,24){$\mathbf{v}_3$}
\put(115,24){$\mathbf{v}_4$}
\end{picture}
\caption{Left: Polyhedron $P_1$. Middle: Polyhedron: $Q$. Right: Polyhedron $P_2$. The vector $v_5$ is in all figures below the polytope, therefore it is not visible. 
}
\label{ExampleP1dim3}
\end{figure}

The normal fans $\Delta _1$, $\Delta _2$ and $\Delta _3$ of $P_1, P_2$ and $Q$ are depicted in Figure \ref{ExampleNormalFans}

\begin{figure}[!h]
\centering
\begin{picture}(200,100)
\put(-50,0){\includegraphics[width=0.25\textwidth]{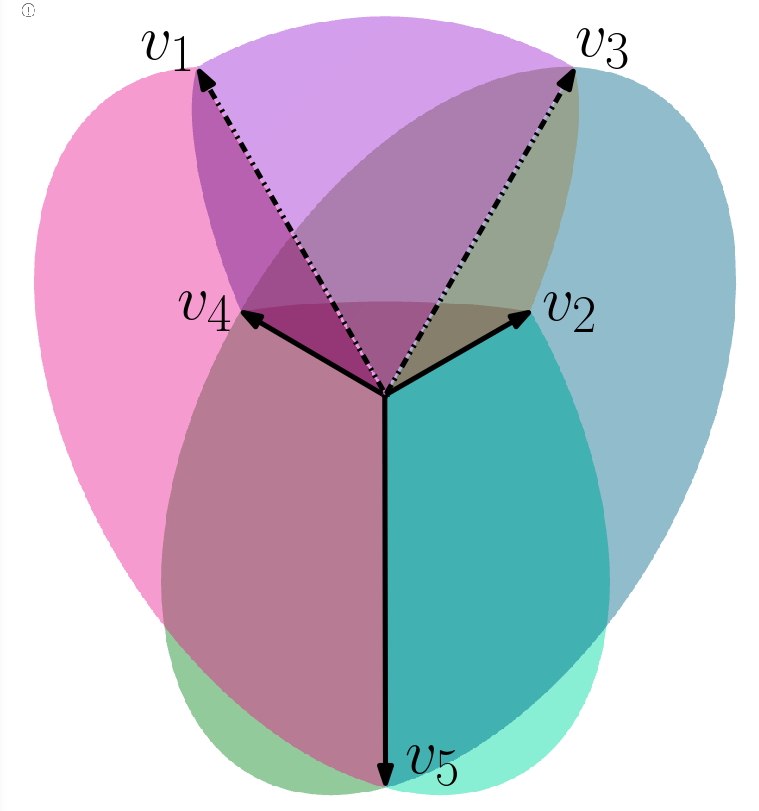}}
\put(55,0){\includegraphics[width=0.24\textwidth]{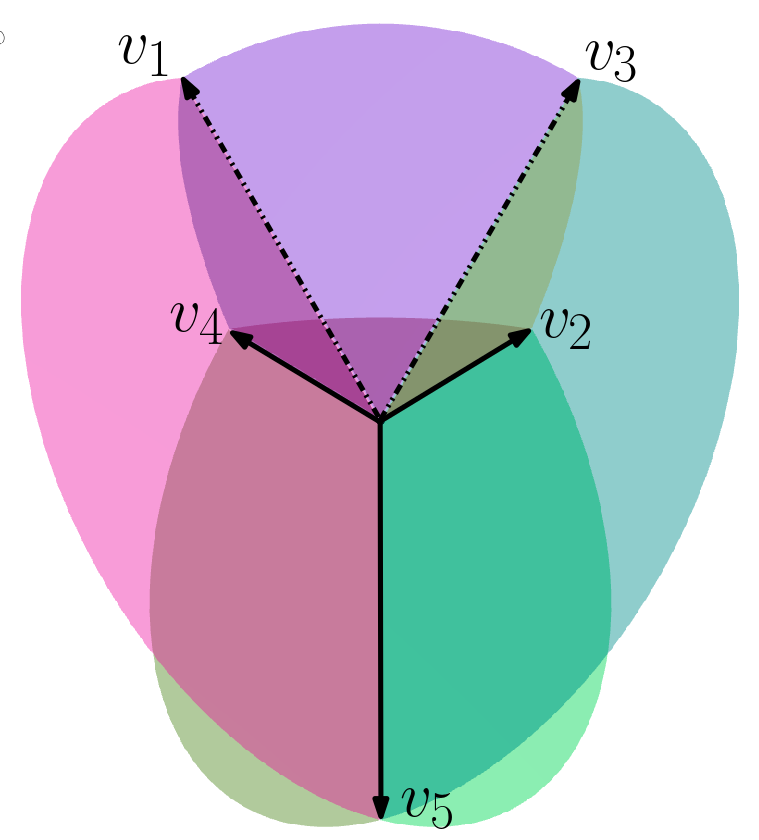}}
\put(160,0){\includegraphics[width=0.24\textwidth]{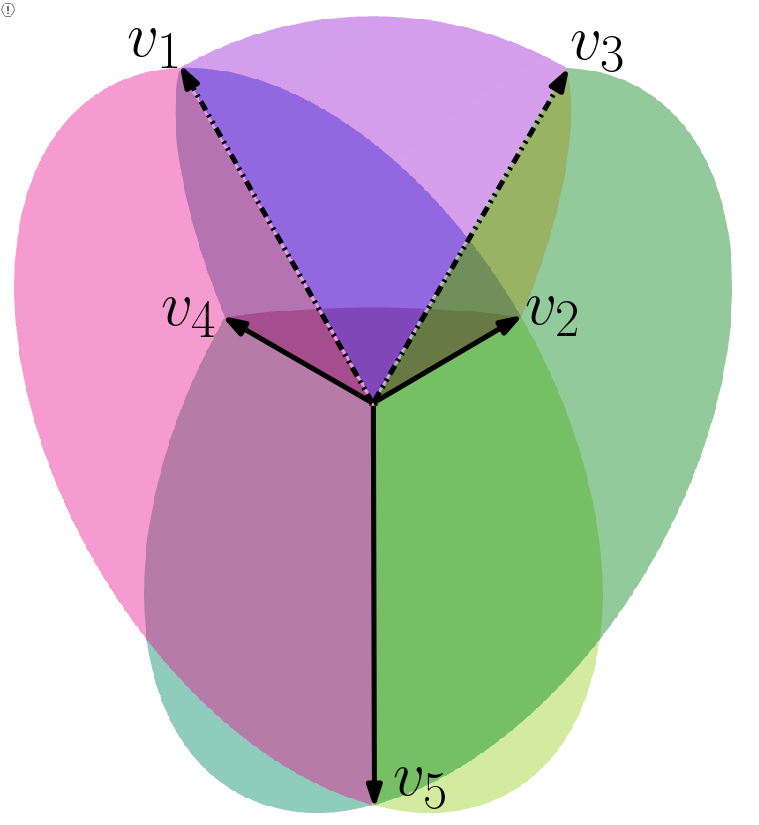}}
\end{picture}
\caption{Left: Normal fan of $P_1$. Middle: Normal fan of $Q$. Right: Normal fan of $P_2$. }
\label{ExampleNormalFans}
\end{figure}

The full-dimensional cones of $\Delta _1$ are 
\begin{align*}
\beta_1 = \pos(\{\mathbf{v}_1,\mathbf{v}_3,\mathbf{v}_5\}), \beta_2 = \pos(\{\mathbf{v}_2,\mathbf{v}_3,\mathbf{v}_5\}), \beta_3 = \pos(\{\mathbf{v}_1,\mathbf{v}_4,\mathbf{v}_5\}),\\
\beta_4 = \pos(\{\mathbf{v}_2,\mathbf{v}_4,\mathbf{v}_5\}), \beta_5 = \pos(\{\mathbf{v}_1,\mathbf{v}_3,\mathbf{v}_4\}), \beta_6 = \pos(\{\mathbf{v}_2,\mathbf{v}_3,\mathbf{v}_4\}).
\end{align*}

There are nine pairs of neighboring full-dimensional cones, each corresponding to an edge in $P_1$. Applying Proposition~\ref{prop:wallcrossing} to the generators of each such pair yields
$$\mathcal{P}(\Delta _1) = \{ \mathbf{h} \in \mathbb{R}^5 : h_1 + h_2 - h_3 - h_4 \geq 0, h_3+h_4+2h_5 \geq 0 \} \, .$$
Similarly, the maximal cells of $\Delta _2$ are
\begin{align*}
\gamma_1 = \beta_i \text{ for } i \in [4], \gamma_5 = \pos(\{\mathbf{v}_1,\mathbf{v}_2,\mathbf{v}_3\}), \gamma_6 = \pos(\{\mathbf{v}_1,\mathbf{v}_2,\mathbf{v}_4\})
\end{align*}
We see that the normal fans of $P_1$ and $P_2$ have the first four maximal cones in common but differ in the remaining two. For example, the maximal cone $\beta _5$ is contained in the normal fan of $P_1$ as the facets with normals $\mathbf{v}_1,\mathbf{v}_3$ and $\mathbf{v}_4$ intersect in the vertex $p_5$. In contrast, the corresponding facets in $P_2$ do not intersect and thus $\beta _5$ is not contained in the normal fan of $P_2$. The deformation cone $\mathcal{P}(\Delta _2)$ has the inequality description
$$\mathcal{P}(\Delta _2) = \{\ \mathbf{h} \in \mathbb{R}^5 : h_3 + h_4 - h_1 - h_2 \geq 0, h_1+h_2+2h_5 \geq 0 \} \, .$$
Every full-dimensional, simple polytope with facet directions $V$ has normal fan $\Delta _1$ or $\Delta _2$. Therefore,
\[
\clir(V) \ = \ \mathcal{P}(\Delta _1)\cup \mathcal{P}(\Delta _2) \ = \{\ \mathbf{h} \in \mathbb{R}^5 :  h_1+h_2+2h_5 \geq 0, h_3+h_4+2h_5\geq 0 \} \, .
\]

From the inequality descriptions we see that $Q$ is contained in the boundaries of both $\mathcal{P}(\Delta _1)$ and $\mathcal{P}(\Delta _2)$ but not in their interiors $\mathcal{T}(\Delta _1)$ and $\mathcal{T}(\Delta _2)$. Indeed, $Q$ can be obtained from both $P_1$ and $P_2$ by pushing in facets until the edge $p_5$ and $p_6$ degenerates to a point. This corresponds to the fact that the normal fan $\Delta _3$ refines both $\Delta _1$ and $\Delta _3$; the union of $\beta _5$ and $\beta _6$ (which equals the union of $\gamma _5$ and $\gamma _6$) forms a maximal cell in $\Delta _3$. This cell is not simplicial and indeed, $Q$ is not simple.

\end{example}

For every vector $\mathbf{u}\in \mathbb{R}^d$ we define a vector $[\mathbf{u}]^\Delta \in \mathbb{R}^n$ in the following way: let $\sigma$ be the unique cone in $\Delta$ such that $\mathbf{u}$ is contained in $\relint \sigma$. If $\Delta$ is clear from the context we simply write $[\mathbf{u}]$. Let $\mathbf{I}_\sigma$ be the set of indices of generators of $\sigma$, that is, $\sigma = \pos \{\mathbf{v}_i \colon i\in I_\sigma \}$, and let $\mathbf{u}=\sum _{k\in I_\sigma}\lambda _k \mathbf{v}_k$. Note that $|\mathbf{I}_\sigma|\leq d$ since $\Delta$ is simplicial. Then 
\[
[\mathbf{u}]_i=\begin{cases}
\lambda _i & \text{ if } i\in I_\sigma \, , \\
0 & \text{ otherwise.}
\end{cases}
\]
With this notation, by Lemma~\ref{lem:hPlinear}, the evaluation of the support function of $P=P(\mathbf{h})$ at $\mathbf{u}$ equals
\[
h_P(\mathbf{u})=\langle \mathbf{h},[\mathbf{u}]\rangle \, .
\]
We observe that $\mathbf{u}\mapsto [\mathbf{u}]$ is a continuous function. Thus the maximum
\[
c^\Delta = \max _{\mathbf{u}\in \mathbb{S}^{d-1}} \max _{1\leq i\leq n}[\mathbf{u}]_i
\]
exists. The Hausdorff distance between two polytopes $P(\mathbf{h})$ and $P(\mathbf{h}')$ can be bounded from above in the following way.
\begin{lemma}\label{lem:Hausdorffdist} For all polytopes $P(\mathbf{h})$ and $P(\mathbf{h}')$ in $\mathcal{P}(\Delta)$
\[
d_H(P(\mathbf{h}),P(\mathbf{h}')) \leq \sqrt{d}c^\Delta \| \mathbf{h}-\mathbf{h}'\| \, .
\]
\end{lemma}
\begin{proof}
For all $\mathbf{u}\in \mathbb{S}^{d-1}$
\[
|h_{P(\mathbf{h})}(\mathbf{u})-h_{P(\mathbf{h}')}(\mathbf{u})|= |\langle(\mathbf{h}-\mathbf{h}'),[\mathbf{u}]\rangle| \leq \|[\mathbf{u}]\|\| \mathbf{h}-\mathbf{h}'\|  \leq \sqrt{d}c^\Delta \| \mathbf{h}-\mathbf{h}'\| \, 
\]
since $[\mathbf{u}]$ has at most $d$ non-zero entries. In particular,
\[
d_H(P(\mathbf{h}),P(\mathbf{h}'))=\max_{\mathbf{u}\in \mathbf{S}^{d-1}} |h_{P(\mathbf{h})}(\mathbf{u})-h_{P(\mathbf{h}')}(\mathbf{u})|\leq \sqrt{d}c^\Delta \| \mathbf{h}-\mathbf{h}'\| \, .
\]
\end{proof}

\section{Structural results}\label{sec:structuralresults}
\subsection{Geometry of the solution set}\label{subsec:solutionset}
Given a data set of pairs of directions and support function evaluations 
\[
\{(\mathbf{u}^{(i)},y^{(i)}) \colon y^{(i)}=h_P(\mathbf{u}^{(i)}) + \varepsilon ^{(i)} \}_{i=1}^m
\]
we call $U \ = \ (\mathbf{u}^{(1)},\ldots, \mathbf{u}^{(m)})^\mathsf{T}$ the \textbf{matrix of directions} and $\mathbf{y}=(y^{(1)},\ldots, y^{(m)})^\mathsf{T}$ the \textbf{vector of support function evaluations}. For a given simplicial normal fan $\Delta$ let furthermore $A_U^\Delta \in \mathbb{R}^{m\times n}$ be the matrix
\[
A_U^\Delta = \left([\mathbf{u}^{(1)}]^\Delta,\ldots, [\mathbf{u}^{(m)}]^\Delta\right)^\mathsf{T}.
\]
Observe that $A_U^\Delta$ is sparse in the sense that every row contains at most $d$ nonzero entries. If $\Delta$ is clear from the context we write $A_U$. The following theorem shows that given a fixed simplicial fan $\Delta$ the least-squares estimator $\hat{P}^\Delta(U,\mathbf{y})$ is equal to the solution set of a convex quadratic program.
\begin{theorem}\label{thm:main}
For a data set $\{(\mathbf{u}^{(i)},y^{(i)})\}_{i=1}^m\subset \R^d \times \R$ and a given polytopal simplicial fan $\Delta$
\[
\hat{P}^\Delta(U,\mathbf{y}) = \argmin _{P(\mathbf{h})\in \mathcal{P}(\Delta)}\| A_U\mathbf{h}-\mathbf{y}\| \, .
\]
Equivalently,
\[
\hat{P}^\Delta(U,\mathbf{y}) \ = \ \mathcal{P}(\Delta)\cap \{\mathbf{h} \in \mathbb{R}^n \colon A_U\mathbf h = \hat{\mathbf{y}}\} \, ,
\]
where $\hat{\mathbf{y}}\in A_U \mathcal{P}(\Delta)$ is the unique point of minimal distance to $\mathbf{y}$. In particular, the solution set $\hat{P}^\Delta(U,\mathbf{y})$ is a non-empty polyhedron in the parameter space $\mathbb{R}^n$.
\end{theorem}
\begin{proof}
Let $\mathbf{v}_1,\ldots, \mathbf{v}_n$ be the ray generators of $\Delta$. Then
\[
\mathbf{u}^{(i)}=\sum _{k= 1}^n[\mathbf{u}^{(i)}]_k \mathbf{v}_k \, .
\]
By Lemma~\ref{lem:hPlinear}, $h_P$ is linear on every cone $\sigma$ in $\Delta$. Therefore we have
\[
h_P (\mathbf{u}^{(i)})=\sum _{k=1}^n [\mathbf{u}^{(i)}]_k h_P( \mathbf{v}_k)=\sum _{k=1}^n(A_U)_{ik} h_k = (A_U\mathbf{h})_i \, .
\]
Thus it follows that
\begin{eqnarray}
\hat{P}^\Delta(U,\mathbf{y})&=&\argmin _{\mathbf{h}\in \mathcal{P}(\Delta)} \sum _{i=1}^m \left((A_U\mathbf{h})_i-y^{(i)}\right)^2\\
&=& \argmin _{P(\mathbf{h})\in \mathcal{P}(\Delta)}\| A_U\mathbf{h}-\mathbf{y}\|^2 \nonumber \, . \label{eq:AU}
\end{eqnarray}
Since $\mathcal{P}(\Delta)$ is a convex polyhedral cone so is its image under multiplication with $A_U$. By convexity, there is a unique point $\hat{\mathbf{y}}\in A_U\mathcal{P}(\Delta)$ such that $\|\hat{\mathbf{y}}-\mathbf{y}\|$ is minimal. The least-squares estimator $\hat{P}^\Delta(U,\mathbf{y})$ equals $\argmin _{P(\mathbf{h})\in \mathcal{P}(\Delta)}\| A_U\mathbf{h}-\mathbf{y}\|$ by Equation~\eqref{eq:AU}. Thus, a polyhedron $P=P(\mathbf{h})\in \mathcal{P}(\Delta)$ is in the solution set $\hat{P}^\Delta(U,\mathbf{y})$ if and only if $A_U \mathbf{h}=\hat{\mathbf{y}}$. It follows that $\hat{P}^\Delta(U,\mathbf{y}) = \mathcal{P}(\Delta)\cap \{\mathbf{h} \in \mathbb{R}^n \colon A_U\mathbf h = \hat{\mathbf{y}}\}$. In particular,  $\hat{P}^\Delta(U,\mathbf{y})$ is an intersection of a polyhedral cone and an affine hyperplane and therefore it is a polyhedron. It is non-empty since $\hat{\mathbf{y}}$ is contained in $A_U \mathcal{P}(\Delta)$.
\end{proof}
If the data provided are noiseless then the nearest point $\hat{\mathbf{y}}$ equals $\mathbf{y}$ and we obtain the following expression for the solution set.
\begin{corollary}
For a given polytopal simplicial fan $\Delta$ and a set of noiseless data $\{(\mathbf{u}^{(i)},y^{(i)})\}_{i=1}^m\subset \R^d \times \R$ 
\[
\hat{P}^\Delta(U,\mathbf{y}) = \{\mathbf{h}\in \mathcal{P}(\Delta)\colon A_U \mathbf{h}=\mathbf{y} \} \, .
\]
\end{corollary}

From Theorem~\ref{thm:main} it follows that for a given polytopal simplicial fan $\Delta$ the solution set $\hat{P}^\Delta(U,\mathbf{y})$ is a non-empty polyhedron. The next result characterizes when the solution set is unbounded.

\begin{proposition}\label{prop:charunbounded}
For a data set $\{(\mathbf{u}^{(i)},y^{(i)})\}_{i=1}^m\subset \R^d \times \R$ and a given polytopal simplicial fan $\Delta$ the following statements are equivalent.
\begin{itemize}
\item[(i)] The solution set $\hat{P}^\Delta(U,\mathbf{y})$ is unbounded.
\item[(ii)] There is a polyhedron $P=P(\mathbf{h})\in \mathcal{P}(\Delta)$, $\mathbf{h}\neq \mathbf{0}$, such that 
\[h_P(\mathbf{u}^{(i)})=0 \quad \text{for all }i=1,\ldots, m \, . \]
\item[(iii)] The cone $\ker A_U \cap \mathcal{P}(\Delta)$ is non-trivial.
\end{itemize}
\end{proposition}
\begin{proof}
For all $P=P(\mathbf{h})\in \mathcal{P}(\Delta)$, $h_P(\mathbf{u}^{(i)})=(A_U\mathbf{h})_i$ as argued in the proof of Theorem~\ref{thm:main}. Thus $h_P(\mathbf{u}^{(i)})=0$ for all $i$ if and only if $\mathbf{h}\in \ker A_U$. This proves the equivalence of (ii) and (iii). By Theorem~\ref{thm:main}, $\hat{P}^\Delta(U,\mathbf{y})$ is a non-empty polyhedron. Let $\mathbf{g}$ be an element in $\hat{P}^\Delta(U,\mathbf{y})$ and let $\hat{\mathbf{y}}$ be the unique point of minimal distance to $\mathbf{y}$ in $A_U\mathcal{P}(\Delta)$. If we assume that condition (iii) is satisfied then there exists $\mathbf{h}\in \mathcal{P}(\Delta)$, $\mathbf{h}\neq \mathbf{0}$, such that $A_U \mathbf{h}=\mathbf{0}$. It follows that $A_U(\mathbf{g}+t\mathbf{h})=A_U\mathbf{g}+tA_U\mathbf{h}=\hat{\mathbf{y}}+t\mathbf{0}=\hat{\mathbf{y}}$ for all $t\geq 0$. Furthermore, since $\mathcal{P}(\Delta)$ is a convex cone, $\mathbf{g}+t\mathbf{h}$ is in $\mathcal{P}(\Delta)$ for all $t\geq 0$. In particular, $\mathbf{g}+t\mathbf{h}$ is in $\hat{P}^\Delta(U,\mathbf{y})$ by Theorem~\ref{thm:main} for all $t\geq 0$ and therefore $\hat{P}^\Delta(U,\mathbf{y})$ is unbounded, that is, condition (i) is satisfied. Vice versa, if we assume that $\hat{P}^\Delta(U,\mathbf{y})$ is unbounded, then, by convexity, there exists a $\mathbf{g}\in \hat{P}^\Delta(U,\mathbf{y})$ and an $\mathbf{h}\in \mathbb{R}^n$ such that $\mathbf{g}+t\mathbf{h}\in \mathcal{P}(\Delta)$ and $A_U(\mathbf{g}+t\mathbf{h})=\hat{\mathbf{y}}$ for all $t\geq 0$. By linearity it follows that $\mathbf{h}\in \ker A_U$. Furthermore, since $\mathcal{P}(\Delta)$ is a closed convex cone, $1/t (\mathbf{g}+t\mathbf{h})\in \mathcal{P}(\Delta)$ for all $t\geq 0$ and thus also $\lim _{t\rightarrow \infty}1/t (\mathbf{g}+t\mathbf{h})=\mathbf{h}$ is in $\mathcal{P}(\Delta)$. This shows that (iii) is satisfied.
\end{proof}
\begin{example}\label{ex:notunique}
Let $\Delta$ be the hexagonal fan as in Example~\ref{ex:1} with generators $\mathbf{v}_1,\ldots, \mathbf{v}_6$. For all $i\in \mathbb{Z}/6\mathbb{Z}$ let $\mathbf{u}^{(i)}=\mathbf{v}_i+\mathbf{v}_{i+1}$ and let $y^{(i)}=2$. We obtain
\[
A_U \ = \ \begin{bmatrix}
1&1&0&0&0&0\\
0&1&1&0&0&0\\
0&0&1&1&0&0\\
0&0&0&1&1&0\\
0&0&0&0&1&1\\
1&0&0&0&0&1
\end{bmatrix}
\, .
\]
The matrix $A_U$ is of rank $5$ and its kernel is spanned by $\mathbf{z}=(1, -1, 1, -1, 1, -1)^\mathsf{T}$. Since the row sums of $A_U$ are all equal to $2$ we obtain that $P(\mathbf{1}^\mathsf{T})$ is contained in the solution set. By Theorem~\ref{thm:main} the set of solutions is equal to
\[
\{\mathbf{h}\in \mathbb{R}^6\colon h_{i}+h_{i+2}-h_{i+1}\geq 0\}\cap \{\mathbf{1}^\mathsf{T}+\lambda \mathbf{z}^\mathsf{T} \colon \lambda\in \mathbb{R}\} \, .
\]
Thus, a polytope $P(\mathbf{1}^\mathsf{T}+\lambda \mathbf{z}^\mathsf{T})$ is in $\mathcal{P}^\Delta (U,\mathbf{y})$ if and only if $\lambda$ satisfies the inequalities
\begin{eqnarray*}
(1+\lambda)+(1+\lambda)&\geq& 1-\lambda \, , \\
(1-\lambda)+(1-\lambda)&\geq& 1+\lambda \, .
\end{eqnarray*}
These inequalities are equivalent to $-\frac{1}{3}\leq \lambda \leq \frac{1}{3}$. Therefore,
\[
\hat{P}^\Delta(U,\mathbf{y})=\conv \left\{ \left(\frac{4}{3}, \frac{2}{3}, \frac{4}{3}, \frac{2}{3}, \frac{4}{3}, \frac{2}{3}\right), \left(\frac{2}{3}, \frac{4}{3}, \frac{2}{3}, \frac{4}{3}, \frac{2}{3}, \frac{4}{3}\right)\right \} \, 
\]
is a segment. It can be seen that its two endpoints $P\left(\mathbf{1}^\mathsf{T}\pm 1/3\ \mathbf{z}^\mathsf{T}\right)$ are triangles.
\end{example}
If we remove the restriction on the normal fan and consider more generally the set of all polytopes with fixed facet directions $V=\{\mathbf{v}_1,\ldots, \mathbf{v}_n\} \subset \mathbb{R}^d$ we obtain the following corollary from Theorem~\ref{thm:main}.
\begin{corollary}\label{cor:piecewisequadratic}
Let $V=\{\mathbf{v}_1,\ldots, \mathbf{v}_n\}$ be positively spanning vectors in $\mathbb{R}^d$. Then for a given data set $\{(\mathbf{u}^{(i)},y^{(i)})\}_{i=1}^m\subset \R^d \times \R$ the least-squares estimator
\[
\argmin _{\mathbf{h}\in \clir (V)} \frac{1}{m}\sum _{i=1}^m \left(h_{P(\mathbf{h})}(\mathbf{u}^{(i)})-y^{(i)}\right)^2
\]
is given by a piecewise quadratic program, that is, the objective function is piecewise quadratic.
\end{corollary}
\begin{proof}
In the proof of Theorem~\ref{thm:main} we argued that for fixed simplicial normal fan $\Delta$ and for all $\mathbf{h}\in \mathcal{P}(\Delta)$
\[
(h_{P(\mathbf{h})} (\mathbf{u}^{(i)}))_{i=1,\ldots m} \ = \ A_U^{\Delta} \mathbf{h} \, .
\]
Further, the set $\clir (V)$, which parametrizes all polytopes with facet directions $V$, is subdivided by the family of all deformation cones $\mathcal{P}(\Delta)$ where $\Delta$ is a simplicial polytopal fan with ray generators $V$. Thus, for fixed directions $U=(\mathbf{u}^{(1)},\ldots, \mathbf{u}^{(m)})$,
\begin{eqnarray*}
\clir (V) &\rightarrow& \mathbb{R}^m\\
\mathbf{h}&\mapsto&(h_{P(\mathbf{h})} (\mathbf{u}^{(i)}))_{i=1,\ldots m}
\end{eqnarray*}
defines a piecewise linear map where the regions of linearity are given by the deformation cones $\mathcal{P}(\Delta)$. In particular, $\frac{1}{m}\sum _{i=1}^m \left(h_{P(\mathbf{h})}(\mathbf{u}^{(i)})-y^{(i)}\right)^2$ is a piecewise quadratic function.
\end{proof}

The number of regions of linearity in the piecewise quadratic program in Corollary~\ref{cor:piecewisequadratic} is upper bounded by the number of deformation cones, which, in general, can be large. However, in special situations there might be fewer regions of linearity. For example, if $m = n$ and $\mathbf{u}^{(i)} = \mathbf{v}_i$ for all $i=1, \ldots, m$, then $A_U^{\Delta}$ is equal to the identity matrix for all possible normal fans $\Delta$.  Therefore, there is only one region of linearity and the program becomes a quadratic program.

\subsection{Uniqueness}\label{subsec:uniqueness}
In this section let again $\Delta$ be a fixed simplicial polytopal fan in $\mathbb{R}^d$. We consider the case that $\hat{P}^\Delta(U,\mathbf{y})$ consists of only one element. In this case the reconstruction is unique and we also call $\hat{P}^\Delta(U,\mathbf{y})$ unique. From Theorem~\ref{thm:main} it follows that  $\hat{P}^\Delta(U,\mathbf{y})$ is unique for all $\mathbf{y}\in \mathbb{R}^d$ if $\rank A_U=n$. For $\mathbf{y}\in \relint A_U \mathcal{P}(\Delta)$ the reverse implication is true as well. 
\begin{proposition}\label{prop:uniquenessy}
For all $\mathbf{y}\in \relint A_U \mathcal{P}(\Delta)$ the solution set $\hat{P}^\Delta(U,\mathbf{y})$ is unique if and only if $\rank A_U =n$.
\end{proposition}
\begin{proof}
 To see the reverse implication, we observe that, by convexity and since $\mathcal{P}(\Delta)$ is full-dimensional, for all $\mathbf{y}\in \relint A_U \mathcal{P}(\Delta)$ there exists a $\mathbf{h}_0\in \inte \mathcal{P}(\Delta)$ such that $A_U \mathbf{h}_0 = \mathbf{y}$. If $\rank A_U<n$ then multiplication with $A_U$ is not injective and the affine space $\{\mathbf{h}\in \mathbb{R}^n \colon A_U \mathbf{h}=\mathbf{y}\}$ has dimension at least $1$. This space contains $\mathbf{h}_0$ and thus intersects $\mathcal{P}(\Delta)$ in its interior. Since $\hat{P}^\Delta(U,\mathbf{y})=\mathcal{P}(\Delta)\cap \{\mathbf{h}\in \mathbb{R}^n \colon A_U \mathbf{h}=\mathbf{y}\}$ by Theorem ~\ref{thm:main}, the solution set $\hat{P}^\Delta(U,\mathbf{y})$ contains more than one element in this case.
\end{proof}
Let
\[
\mathcal{U} \ = \ \{U\in \mathbb{R}^{m\times d} \colon | \hat{P}^\Delta (U,\mathbf{y}) | =1 \text{ for all }\mathbf{y}\in\mathbb{R}^m \}  \subset \mathbb{R}^{m\times d}  \, 
\]
denote the set of all $U\in \mathbb{R}^{m\times d}$ such that $\hat{P}^\Delta(U,\mathbf{y})$ is unique for all $\mathbf{y}\in \mathbb{R}^m$. Observe that, by Theorem~\ref{thm:main} and Proposition~\ref{prop:uniquenessy}, a matrix $U\in \mathbb{R}^{m\times d}$ is in $\mathcal{U}$ if and only if $\rank A_U=n$. 

Further, for every matrix $U\in \mathbb{R}^{m\times d}$ we define a bipartite graph $G_U$ with vertex set $[n]\sqcup [m]$ in the following way. If $\mathbf{u}^{(1)},\ldots, \mathbf{u}^{(m)}$ denote the rows of $U$, then $i\in [n]$ and $j\in [m]$ are connected by an edge if and only if $[\mathbf{u}^{(j)}]_i >0$.

The following theorem provides a combinatorial characterization of the set $\mathcal{U}$.

\begin{theorem}\label{thm:matching}
The set $\mathcal{U}$ has non-empty interior for $m\geq n$. Moreover, for a generic matrix $U\in \mathbb{R}^{m\times d}$, $U\in \mathcal{U}$ if and only if $G_U$ has a matching of size $n$.
\end{theorem}
Here and in the following, \textbf{generic} means that the statement is correct except for a subset of matrices of Lebesgue measure zero. In the following we first discuss geometric properties of the set $\mathcal{U}$. These form the base of our proof of Theorem \ref{thm:matching} below.

A subset of $\mathbb{R}^d$ is called \textbf{semi-algebraic} if it  is a finite boolean combination of sets of the form $\{\mathbf{x} \in \mathbb{R}^d : f_i(\mathbf{x}) \geq 0 \text{ for } i= 1, \ldots, m\}$, where $f_1, \ldots, f_n$ are polynomials in $\mathbb{R}[x_1, \ldots, x_d]$. 

\begin{proposition}\label{prop:injective}
The set $\mathcal{U}\subseteq \mathbb{R}^{m\times d}$ is semi-algebraic.
\end{proposition}
\begin{proof}
 The set $\mathcal{U}$ is the preimage of the set of all $m\times n$ matrices of rank $n$ under the map $\mathbb{R}^{m\times d}\rightarrow \mathbb{R}^{m\times n}, U\mapsto A_U$. If $m<n$ then $\rank A_U<n$ for all $U$ and thus $\mathcal{U}$ equals the empty set which is semi-algebraic. If $m\geq n$, the set of all $m\times n$ matrices of rank $n$ is equal to the complement of the set of all matrices of rank at most $n-1$ which is exactly the set of all matrices for which all maximal minors vanish. Since this is a semi-algebraic set so is its complement. It remains to show that $U\mapsto A_U$ is a semi-algebraic homomorphism. Indeed, by definition, the map $\mathbf{u}\mapsto [\mathbf{u}]$ is piecewise linear, and thus
\[
U \ = \ \left(\mathbf{u}^{(1)},\ldots, \mathbf{u}^{(m)}\right)^\mathsf{T} \longmapsto A_U \ = \ \left([\mathbf{u}^{(1)}],\ldots, [\mathbf{u}^{(m)}]\right)^\mathsf{T}
\]
is piecewise linear as well.
\end{proof}

The set of all graphs $G$ such that there exists a matrix $U\in \mathbb{R}^{m\times d}$ for which $G=G_U$ we denote by $\mathcal{G}$. Note that not every bipartite graph on $[n]\sqcup [m]$ arises in this way.
From the definition we obtain that
\[
\mathcal{G} \ = \ \{ G \subseteq K_{n,m} \colon \pos (\{v_i \colon i\in N(j)\})\in \Delta \text{ for all } j\in [m] \}
\]
where $K_{n,m}$ denotes the complete bipartite graph with vertex set $[n]\sqcup [m]$ and $N(j)\subseteq [n]$ the set of all neighbors of $j\in [m]$.

For every bipartite graph $G$ in $\mathcal{G}$ we define $
U_G = \{U\in \mathbb{R}^{m\times d} \colon G_U = G \}$. By definition, the set of all $U_G$ for $G\in \mathcal{G}$ partitions $\mathbb{R}^{m\times d}$. Moreover, we see that
\[
U_{G_{U}} \ = \ \relint \sigma (\mathbf{u}^{(1)}) \times \cdots \times \relint \sigma (\mathbf{u}^{(m)})
\]
where $U=(\mathbf{u}^{(1)},\ldots ,\mathbf{u}^{(1)})^\mathsf{T}$ and $\sigma (\mathbf{u}^{(i)})$ denotes the carrier of $\mathbf{u}^{(i)}$. In particular, $U_G$ is a relatively open polyhedral cone in $\mathbb{R}^{m\times d}$ for all $G\in \mathcal{G}$. Indeed, the closed cones $\{\overline{U_G}\subset \mathbb{R}^{m\times d} \colon G\in \mathcal{G}\}$ coincide with the cells in the polyhedral subdivision $\Delta^m$ of $\mathbb{R}^{m\times d}$.

We will now focus on the set of graphs $G$ in $\mathcal{G}$ for which a matching of size $n$ exists. This will lead us to a characterization of the uniqueness of the reconstruction $\hat{P}^\Delta(U,\mathbf{y})$.

Let $\Upsilon$ be the collection of all closed polyhedral cones $\overline{U_G}$ for which $G$ does not have a matching of size $n$.
\begin{lemma}
The collection $\Upsilon$ is a polyhedral subcomplex of $\Delta ^m$.
\end{lemma}
\begin{proof}
It suffices to show that for every $\overline{U_G}$ in $\Upsilon$ also every facet of $\overline{U_G}$ is contained in $\Upsilon$. Let $U=(\mathbf{u}^{(1)},\ldots ,\mathbf{u}^{(m)})^\mathsf{T}$ be a matrix such that $G_U=G$. Then every facet of $\overline{U}_G$ is of the form
\[
F=\sigma (\mathbf{u}^{(1)}) \times \cdots \times \sigma (\mathbf{u}^{(i-1)}) \times H\times \sigma (\mathbf{u}^{(i+1)}) \times \cdots \times  \sigma (\mathbf{u}^{(m)})
\]
where $H$ is a facet of $\sigma (\mathbf{u}^{(i)})$ for some $i$. Since $\Delta$ is a simplicial fan, there is a unique generator $\mathbf{v}_k$ of $\sigma (\mathbf{u}^{(i)})$ that is not a generator of $H$. In particular, $[\mathbf{u}]_k=0$ for all $\mathbf{u}\in H$ and we obtain $F=\overline{U_{G'}}$ where $G'$ arises from $G$ by deleting the edge between $i$ and $k$. In particular, if $G$ does not have a matching of size $n$ then neither does $G'$.

\end{proof}

Let $G_\Delta$ denote the ray-facet incidence graph of the simplicial fan $\Delta$ with $r$ facets $F_1,\ldots, F_r$. Then $G_\Delta$ is a bipartite graph with vertex set $[n]\sqcup [r]$ where $i\in [n]$ and $j\in [r]$ are connected by an edge if $v_i$ is a generator of the facet $F_j$. The following lemma will allow us to determine when the subcomplex $\Upsilon$ is properly contained in $\Delta^m$.
\begin{lemma}\label{lem:rayfacets}
The ray-facet incidence graph $G_\Delta$ of a simplicial fan $\Delta$ has a matching of size $n$.
\end{lemma}
\begin{proof}
By Hall's theorem \cite{Hall}, it suffices to prove that for every subset $W$ of $[n]$, the set of neighbors $N(W)=\{ j\in [r]\colon ij \text{ is an edge of }G_\Delta \text{ for some }i\in W\}$ contains at least as many elements as $W$. Assume that $|W|=k$. Since $\Delta$ is a full-dimensional simplicial fan in $\mathbb{R}^d$, every ray is incident to at least $d$ many facets and thus there are at least $kd$ many edges leaving $W$. If we assume that $|N(W)|<k$, then, by the box principle, there must be at least one facet in $N(W)$ that contains at least $d+1$ many rays in $W$. This is a contradiction since $\Delta$ is simplicial and therefore every facet is incident to exactly $d$ rays. Thus, $|N(W)|\geq k$ and the claim follows.
\end{proof}

The following lemma characterizes when $\Upsilon$ is strictly contained in $\Delta^m$. This will help us understand when $\mathcal{U}$ is non-empty later on.

\begin{lemma}\label{lem:dimensioncomplex}
The polyhedral complex $\Upsilon$ is strictly contained in $\Delta^m$ if and only if $m\geq n$.
\end{lemma}
\begin{proof}
If $m<n$, then $K_{n,m}$ does not have a matching of size $n$ and thus the same is true for all subgraphs. Therefore, $\overline{U_G}\in \Upsilon$ for all $G\in \mathcal{G}$ and thus $\Upsilon =\Delta^n$.

In order to show that $\Upsilon \neq \Delta ^m$ for all $m\geq n$, we consider the ray-facet incidence graph $G_\Delta$. By Lemma~\ref{lem:rayfacets}, $G_\Delta$ has a matching $M$ of size $n$. Let $F_{j_1},\ldots, F_{j_n}$ be the maximal cones of $\Delta$ such that $i$ is connected to $j_i$ in $M$ for all $1\leq i \leq n$. Then for any choice of cones $C_{n+1}, C_{n+2},\ldots, C_{m}\in \Delta$, the graph $G\subseteq K_{n,m}$ such that
\[
\overline{U_G} = F_{j_1}\times \cdots \times F_{j_n}\times C_{n+1}\times \cdots \times C_m
\]
contains a copy of $M$, in particular, $G$ has a matching of size $n$. That is, $\overline{U_G} \not \in \Upsilon$ and thus $\Upsilon \neq \Delta^m$.
\end{proof}

In order to prove Theorem~\ref{thm:matching} we consider the Edmonds matrix $E_G$ of a bipartite graph $G$ with vertex set $[n]\sqcup [n]$ which is an $n\times n$ matrix defined by 
\[
(E_G)_{ij}=\begin{cases} x_{ij} & \text{ if }ij \text{ is an edge of }G\\
0 & \text{ otherwise.}
\end{cases}
\]
Here, $x_{ij}$ are indeterminants. In particular, $\det E_G$ is a polynomial in these indeterminants. The following lemma holds (see \cite[Edmonds' Theorem 7.3]{Rajeev}).
\begin{lemma}\label{lemma:detGnonzero}
Let $G$ be a bipartite graph with vertex set $[n]\sqcup [n]$. Then $G$ has a perfect matching if and only if $\det E_G\neq 0$ (as polynomial).
\end{lemma}
Now we have all ingredients to prove Theorem~\ref{thm:matching}.
\begin{proof}[Proof of Theorem~\ref{thm:matching}]
We will prove the following two statements: (i) $\mathcal{U}\subseteq \mathbb{R}^{m\times d}\setminus \Upsilon$ and (ii) the set $(\mathbb{R}^{m\times d}\setminus \Upsilon)\setminus \mathcal{U}$ has Lebesgue measure zero. By Lemma~\ref{lem:dimensioncomplex}, $\mathbb{R}^{m\times d}\setminus \Upsilon$ is non-empty for $m\geq n$, and therefore has non-empty interior as a complement of a polyhedral complex. Thus, from (i) and (ii) it will follow that $\mathcal{U}$ has non-empty interior. The second statement then follows since $\mathbb{R}^{m\times d}\setminus \Upsilon$ consists by definition of all matrices $U\in \mathbb{R}^{m\times d}$ for which $G_U$ has a matching of size $n$.

To prove claim (i), for all subsets $I\subseteq [m]$ let $A_{U,I}$ denote the submatrix of $A_U$ consisting of all rows indexed by elements in $I$. By definition of $A_U$, the matrix $A_{U,I}$ has a non-zero entry at position $(i,j)$ only if $\mathbf{u}^{(i)}$ is contained in a cell with generator $\mathbf{v}_i$. In particular, if $|I|=n$ then $A_{U,I}$ is an evaluation of the Edmonds matrix of the subgraph $G_U[I\sqcup [n]]$ of $G_U$ induced by the subset of vertices $I\sqcup [n]$. By Proposition~\ref{prop:injective}, $U\in \mathcal{U}$ if and only if $\rank A_U =n$. This is the case if there is a non-zero maximal minor $\det A_{U,J}$ of $A_U$ where $J\subseteq [m]$ is a subset of cardinality $n$. That is, if $U\in \mathcal{U}$ then the Edmonds polynomial of some induced subgraph $G_U[J\sqcup [n]]$ has a non-trivial evaluation $\det A_{U,J}$ and thus is non-zero itself. By Lemma \ref{lemma:detGnonzero}, $G_U[J\sqcup [n]]$ has a perfect matching, and thus $G_U$ has a matching of size $n$. That is, $U\in \mathbb{R}^{m\times d}\setminus \Upsilon$.

To see the claim (ii), we recall that the function $U\mapsto A_U$ is a piecewise linear function, and the regions of linearity are precisely the cells of $\Delta ^m$. It follows that for all subsets $I\subseteq [m]$ with $|I|=n$ the function defined by the maximal minors
\[
p_I(\mathbf{u}^{(1)},\ldots, \mathbf{u}^{(m)}) = \det A_{U,I}
\]
agrees with a polynomial in $U=(\mathbf{u}^{(1)},\ldots, \mathbf{u}^{(m)})^\mathsf{T}$ restricted to $U\in \overline{U_G}$ for all $G\in \mathcal{G}$. To see the claim, it suffices to show that for all graphs $G\in \mathcal{G}$ that have a matching of size $n$ the polynomial $p_I(\mathbf{u}^{(1)},\ldots, \mathbf{u}^{(m)})$ restricted to $U_G$ is not the zero polynomial. If $M$ is a matching of size $n$ of a graph $G$ that connects $i$ with $\sigma (i)\in [m]$ for all $i\in [n]$ then for every $U\in U_G$ the matrix $\tilde{U}$ is contained in $\overline{U_G}$, where $\tilde{U}$ is the matrix where the $\sigma(i)$-th row is replaced by $\mathbf{v}_{i}$. We observe that for this new matrix $A_{\tilde{U},I}$, where $I$ is equal to the set $\{\sigma (1),\sigma (2),\ldots, \sigma (n)\}$, is a permutation matrix and therefore $\det A_{\tilde{U},I}$ is not zero. Therefore $p_I(\mathbf{u}^{(1)},\ldots, \mathbf{u}^{(m)})$ is non-zero and the claim follows.
\end{proof}

\begin{remark}
Observe that Theorem~\ref{thm:matching} only holds for generic matrices of directions $U$. An example of a matrix $U$ such that $G_U$ has a matching of size $n$ but for which $\hat{P}^\Delta (U,\mathbf{y})$ is not unique was given in Example~\ref{ex:notunique}.
\end{remark}

\begin{corollary}\label{thm:main2}
Let $U\in \mathbb{R}^{m\times n}$ be a generic matrix such that for all maximal cells $\sigma \in \Delta$ there is an $i\in [n]$ such that $\mathbf{u}^{(i)}$ is in $\relint \sigma$. Then $U\in \mathcal{U}$.
\end{corollary}
\begin{proof}
We observe that for all matrices $U$ that satisfy the assumption of Corollary~\ref{thm:main2} the graph $G_U$ contains a copy of the ray-facet incidence graph of $\Delta$. Therefore, by Lemma~\ref{lem:rayfacets}, $G_U$ has a matching of size $n$ and the claim follows thus from Theorem~\ref{thm:matching}.
\end{proof}

Note that if we remove the condition on the normal fan and only fix the facet directions the uniqueness result of Theorem~\ref{thm:matching} does no longer hold in general. An example is given in Section~\ref{subsec:arbitraryuniqueness}. However, there exist special cases in which the uniqueness result still can be guaranteed. For example, if $m=n$ and $\mathbf{u}^{(i)} = \mathbf{v}_i$ for $i=1, \ldots, m$, then $\hat{P}^\Delta(U,\mathbf{y})$ is the unique solution of 
$\argmin _{\mathbf{h}\in \clir(V)}\| \mathbf{h}-\mathbf{y}\|$.

\subsection{Algorithm and complexity}\label{subsec:algorithm}

Given a simplicial polytopal fan $\Delta$ and a set $\{(\mathbf{u}^{(i)},y^{(i)})\}_{i=1}^m$ of possibly noisy support function evaluations, the least-squares estimator $\hat{P}^\Delta (U,\mathbf{y})$ is given as the solution of the convex quadratic program $\argmin _{\mathbf{h}\in \mathcal{P}(\Delta)}\|A_U \mathbf{h}-\mathbf{y}\|$ by Theorem~\ref{thm:main}. Algorithm~\ref{AlgoComplexity} summarizes the steps needed in order to compute $\hat{P}^\Delta (U,\mathbf{y})$. Here, we assume that $\Delta$ is given as a list of ray generating vectors $\mathbf{v}_1,\ldots, \mathbf{v}_n$ together with a list of subsets of $[n]$, each of size $d$, indexing the ray generators of the maximal cells in $\Delta$.

\begin{algorithm}[h]
\caption{Computation of the least-squares estimator $\widehat{P}^{\Delta}(U,\mathbf{y})$}
\label{AlgoComplexity} 
\begin{algorithmic}[1]
\REQUIRE $\Delta,  \{(\mathbf{u}^{(i)}, y^{(i)})\}_{i=1}^m$ 
\ENSURE  $\widehat{P}^\Delta(U,\mathbf{y})$
\STATE $A_U \leftarrow 0\in \mathbf{R}^{m\times n}$
\FOR{$i = 1, \ldots, m$}
\FOR{$\sigma \subseteq [n] \text{ maximal cell of } \Delta$}
\STATE \textit{solve} $\mathbf{u}^{(i)} = \sum_{j \in \sigma} \lambda_{ij} \mathbf{v}_j$
\IF{ $\lambda _{ij}\geq 0$ for all $j\in \sigma$}
\FOR{$j\in \sigma$}
\STATE $(A_U)_{ij} \leftarrow  \lambda_{ij}$ 
\ENDFOR
\ENDIF
\STATE \textbf{break}
\ENDFOR
\ENDFOR
\STATE $p \leftarrow$ number of maximal cells of $\Delta$
\STATE $B \leftarrow 0\in \mathbb{R}^{p^2\times n}$
\STATE $i \leftarrow 1$
\FOR{$\sigma _1, \sigma _2\subseteq [n]\text{ maximal cells in } \Delta$}
\IF{ $|\sigma _1\cap \sigma _{2}|=d-1$}
\STATE \textit{solve} $\lambda_{\sigma _1\setminus \sigma _2} \mathbf{v}_{\sigma _1\setminus \sigma _2} + \lambda_{\sigma _2\setminus \sigma _1} \mathbf{v}_{\sigma _2\setminus \sigma _1} + \sum_{j\in \sigma _1\cap \sigma _2} \lambda _j \mathbf{v}_{j} = 0$

\hspace*{55.5mm}$\lambda_{\sigma _1\setminus \sigma _2}+ \lambda_{\sigma _2\setminus \sigma _1} =2$
\FOR{ $j\in \sigma _1\cup \sigma _2$}
\STATE $B_{ij} \leftarrow \lambda _j$
\ENDFOR
\STATE $i\leftarrow i+1$
\ENDIF
\ENDFOR
\STATE \textit{solve} $\mathsf{argmin}_\mathbf{h} \|A_U \mathbf{h} -\mathbf{y} \|$, $B\mathbf{h}\geq \mathbf{0}$
\RETURN $P(\mathbf{h})= \bigcap _{i=1}^n \{\mathbf{x}\in \mathbb{R}^d\colon \langle \mathbf{x},\mathbf{v}_i\rangle \leq h_i\}$
\end{algorithmic}
\end{algorithm}

In line $1$ to $12$ we determine the matrix $A_U$. For each $i\in [m]$ and every maximal cell $\sigma$ we solve in line $4$ a linear system of equations. Solving each system of linear equations is in $O(d^3)$ because each maximal cell $\sigma$ has at most $d$ generators. By the Upper Bound Theorem~\cite{UpperBound}  the number of maximal cells is at most $\binom{n}{\lfloor d/2\rfloor}$,  in total it takes $O\left(m n^{\lfloor d/2\rfloor} d^3\right)$ to compute $A_U$. Proposition~\ref{prop:wallcrossing} provides the linear inequalities that describe the cone $\mathcal{P}(\Delta)$. In line $13$ to $23$ we determine these inequalities that we summarize as rows of a matrix $B$. For this for every pair of maximal cells we check whether they intersect in a $(d-1)$-dimensional cone, and in this case we solve a system of linear equations in $d+1$ variables. This accounts for a running time of $O(n^dd^3)$.
Finally we determine $\mathbf{h} = \mathsf{argmin}_\mathbf{h} \|A_U \mathbf{h} -\mathbf{y} \|$ with $\mathbf{h} \in P(\Delta)$. This is equivalent to the quadratic program $\mathsf{argmin}_\mathbf{x} \mathbf{x}^\mathsf{T} Q \mathbf{x} + q^\mathsf{T}\mathbf{x}$ subject to $G\mathbf{x} \leq \mathbf{b}$ where
\[
Q = \frac{1}{2}A_U^\mathsf{T}A_U \, , \mathbf{q} = -A_U^\mathsf{T}\mathbf{y}\, , G = -B\, \text{ and }\mathbf{b} = \mathbf{0} \, .
\]
Using the ellipsoid method by Nemirovsky and Yudin \cite{Nemirovsky} the quadratic program can be solved in polynomial time in the length of the input data $L$ \cite{Vavasis2001, Kozlov}. Hence, for fixed dimension $d$ the Algorithm~\ref{AlgoComplexity} has a weakly polynomial running time in $L$.

For fixed dimension $d$, the number of generators $n$ and number of data points $m$ might vary a lot which may influence the running time in practice. For example, the class of generalized permutahedra consists of all deformations of the permutahedron $\Pi _d = \conv \{ (\sigma (1),\ldots, \sigma (d)) \colon \sigma \in S_d\}-(n+1)/2\subset \{\mathbf{x}\in \mathbb{R}^d\colon \sum x_i =0\}$. They are equal to $\mathcal{P}(\Delta)$ where $\Delta$ is the braid arrangement $\bigcup _{i\neq j}\{\mathbf{x}\in \mathbb{R}^d\colon x_i -x_j=0\}$ intersected with the hyperplane $\{\mathbf{x}\in \mathbb{R}^d\colon \sum x_i =0\}$. The number of ray generators of $\Delta$ is $2^d-2$. See, e.g.,~\cite{Postnikov}. In particular, the running time in this case is exponential in the dimension.\\

We end this section by briefly comparing our results with the algorithm of Gardner and Kinderlen~\cite{GardnerNew} for reconstructing an unknown shape from support function evaluations $\{(\mathbf{u}^{(i)},y^{(i)})\}_{i=1}^m$. Instead of using the entries of the support vector as variables they use vectors $\mathbf{x}_1,\ldots, \mathbf{x}_m\in \mathbb{R}^d$ and solve the following constrained least-squares problem.
\begin{eqnarray*}
\min _{\mathbf{x}_1,\ldots, \mathbf{x}_m\in \mathbb{R}^d}&& \sum _{i=1}^m (\mathbf{x}_i^\mathsf{T}\mathbf{u}^{(i)}-y^{(i)})^2\\
\text{subject to}&&\mathbf{x}_j^\mathsf{T}\mathbf{u}^{(i)}\leq \mathbf{x}_i^\mathsf{T}\mathbf{u}^{(i)} \text{ for }1\leq i, j\leq m\, , i\neq j
\end{eqnarray*}
If the vectors $\hat{\mathbf{x}}_1,\ldots, \hat{\mathbf{x}}_m$ form a solution then the unknown shape is approximated by the convex hull of $\hat{\mathbf{x}}_1,\ldots, \hat{\mathbf{x}}_m$. The solution is in general not unique and the facet directions can be arbitrary. However, as was argued in~\cite{GardnerNew}, modifying the algorithm to output the smallest polytope with facet directions $V=\{\mathbf{u}^{(1)},\ldots, \mathbf{u}^{(m)}\}$ containing $\hat{\mathbf{x}}_1,\ldots, \hat{\mathbf{x}}_m$, 
\[
P(\hat{\mathbf{x}}_1^\mathsf{T}\mathbf{u}^{(1)},\ldots,\hat{\mathbf{x}}_m^\mathsf{T}\mathbf{u}^{(m)})=\bigcap _{i=1}^m \left\{\mathbf{x}\in \mathbb{R}^d\colon \mathbf{x}^\mathsf{T}\mathbf{u}^{(i)}\leq \hat{\mathbf{x}}_i^\mathsf{T}\mathbf{u}^{(i)}\right\}
\]
always yields the unique least-squares estimate
\[
\hat{P}(V,\mathbf{y})\ = \ \argmin _{\mathbf{h}\in \clir(V)} \frac{1}{m}\sum _{i=1}^m \left(h_{P(\mathbf{h})}(\mathbf{u}^{(i)})-y^{(i)}\right)^2 \, 
\]
within the class of polytopes with facet directions equal to the directions in which the support function values are given. In particular, if the task is to reconstruct a polytope from support function evaluations in the directions of it facets then the algorithm of Gardner and Kinderlen~\cite{GardnerNew} elegantly circumvents making the constraint $\mathbf{h}\in \clir(V)$ explicit. The complexity of this algorithm is polynomial in the number of data points and the dimension. Especially, depending on the number of facet directions this algorithm can be exponential in the dimension and is thus in this regard not superior to Algorithm~\ref{AlgoComplexity}.

In contrast, if the measurements are not taken in the prescribed facet directions then the modification of the algorithm of Gardner and Kinderlen~\cite{GardnerNew} described above does in general not lead to a good approximation: given measurements $\{(\mathbf{u}^{(i)},y^{(i)})\}_{i=1}^m$, the smallest polytope with facet directions $V=\{\mathbf{v}_1,\ldots, \mathbf{v}_n\}$ containing $\hat{\mathbf{x}}_1,\ldots, \hat{\mathbf{x}}_m$, namely $P(h_1,\ldots, h_n)$ where $h_i=\max _{1\leq j\leq m}\hat{\mathbf{x}}_j^\mathsf{T}\mathbf{v}_i$, depends on the particular solution of $\hat{\mathbf{x}}_1,\ldots, \hat{\mathbf{x}}_m$ and can in general be far from the least-squares estimate, even in the noiseless case. See Figure~\ref{Figure:Gardner}.

\begin{figure}[!h]
\centering
\begin{picture}(100,135)
\put(-100,0){\includegraphics[width=0.8\textwidth]{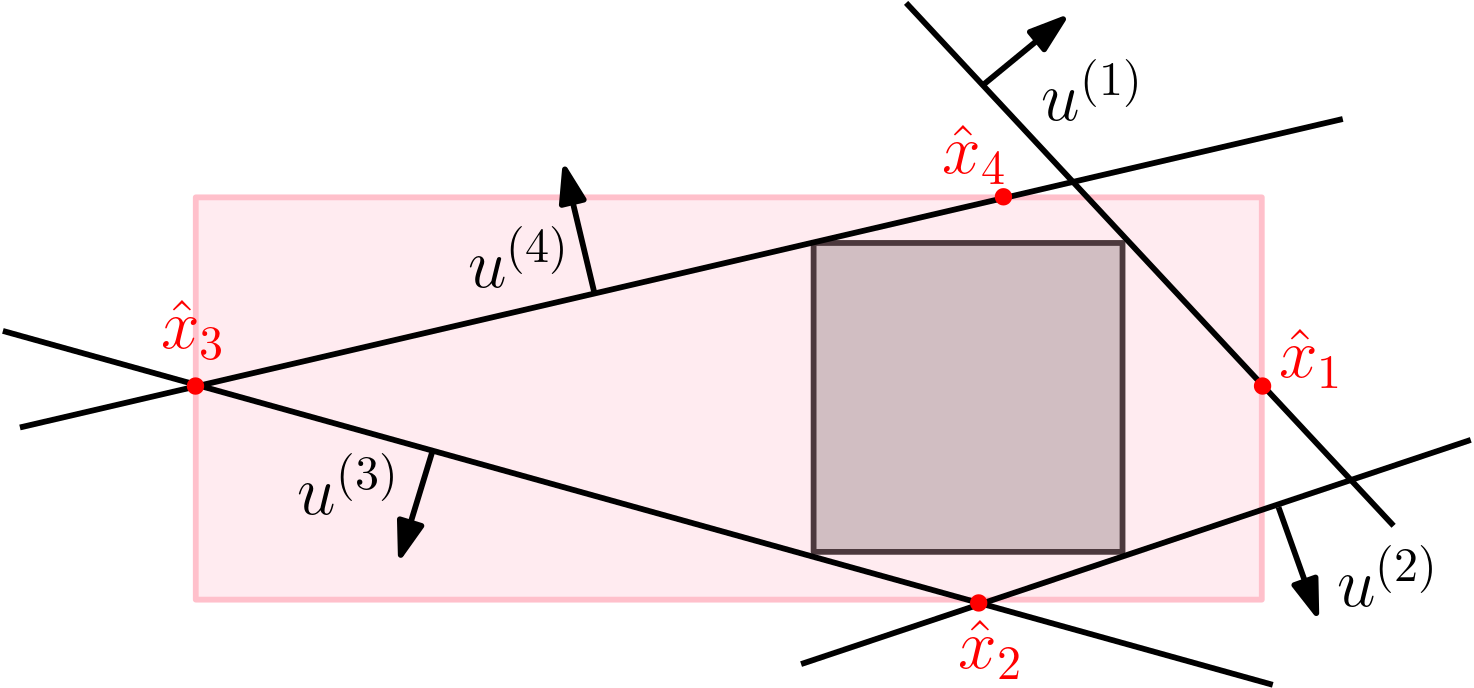}}
\end{picture}
\caption{Input: given $\{(\mathbf{u}^{(i)},y^{(i)})\}_{i=1}^m$ with ´$y^{(i)} = h_P(u^{(i)})$ for unknown input shape $P$ (grey).  Possible output of modified Gardner-Kinderlen algorithm for fixed facet directions $V=\{\pm \mathbf{e}_1,\pm \mathbf{e}_2\}$: smallest rectangle (red) containing $\mathbf{\hat{x}}_i$, $i=1,2,3,4$.}
\label{Figure:Gardner}
\end{figure}

\section{Convergence}\label{sec:convergence}
In this section we consider convergence properties of Algorithm~\ref{AlgoComplexity}. 
In the following let $\mathbf{u}^{(1)},\mathbf{u}^{(2)},\mathbf{u}^{(3)},\ldots$ be a sequence of unit vectors in $\mathbb{S}^{d-1}$ and $P=P(\mathbf{h}_0)$ be a fixed polytope in $\mathcal{P}(\Delta)$. For all $m\geq 1$ let $\hat{P}^\Delta _m(U,\mathbf{y})$ denote the least-squares estimate $\hat{P}^\Delta(U,\mathbf{y})$ of $P$ when only the first $m$ data points $\{(\mathbf{u}^{(i)},y^{(i)}):y^{(i)}=h_P(\mathbf{u}^{(i)})+\varepsilon ^{(i)}\}_{i=1}^m$ are given.

\begin{proposition}
Let $\mathbf{u}^{(1)},\mathbf{u}^{(2)}, \mathbf{u}^{(3)},\ldots \in \mathbb{S}^{d-1}$ be a sequence of directions independently and uniformly chosen. Then

\[
\lim _{m\rightarrow \infty}\mathbb{P}\Big( | \hat{P}^\Delta _m(U,\mathbf{y}) | =1 \text{ for all }\mathbf{y}\in \mathbb{R}^m\Big)=1 \, .
\]

\end{proposition}
\begin{proof}
Let $F_1,\ldots, F_r$ denote the maximal cells of $\Delta$ and let $\mathbf{y}\in \mathbb{R}^m$ be such that $| \hat{P}^\Delta _m(U,\mathbf{y}) | >1$. Then there exists an $i\in [r]$ such that for all $k\in [m]$, $\mathbf{u}^{(k)}\not \in \relint F_i$ almost surely by Corollary~\ref{thm:main2}. Let $E_i (m)$ be the event that $\mathbf{u}^{(k)}\not \in \relint F_i$ for any $k\in [m]$. Let $q_i$ denote the probability that a vector $\mathbf{u}\in \mathbb{S}^{d-1}$ is not contained in $F_i$. Then $q_i$ is equal to $1$ minus the spherical angle of the cone $F_i$ at the origin. 

In particular, since $q_i<1$ for all $i$ and the directions $\mathbf{u}^{(1)},\mathbf{u}^{(2)},\ldots$ are independently chosen, we obtain that
\begin{eqnarray*}
\mathbb{P}\Big( | \hat{P}^\Delta _m(U,\mathbf{y}) | >1 \text{ for some }\mathbf{y}\in \mathbb{R}^m\Big)&\leq& \mathbb{P}(E_1 (m))+\cdots + \mathbb{P}(E_r (m))\\&\leq& q_1^m+\cdots + q_r^m\\
&\to& 0
\end{eqnarray*}
for $m\to \infty$ as desired.
\end{proof}

For every ray generator $\mathbf{v}_j$ of $\Delta$ and every $t>0$ let $C_t (j)$ denote the closed neighborhood of $\mathbf{v}_j/\|\mathbf{v}_j\|$ defined by
\[
C_t (j)=\{\mathbf{x}\in \mathbb{R}^d\colon \|[\mathbf{x}]^\mathsf{T}W-\mathbf{e}_j\|_\infty \leq t\} \, 
\]
where $W$ is the diagonal matrix
\[
W=\begin{pmatrix}
\|\mathbf{v}_1\|&0&\cdots &0\\
0&\|\mathbf{v}_2\|&\cdots &0\\
\vdots&\vdots&\ddots&\vdots\\
0&\cdots & 0& \|\mathbf{v}_n\|
\end{pmatrix} \, .
\]

The main result in this section is the following.

\begin{theorem}\label{thm:convergence}
Let $\mathbf{u}^{(1)},\mathbf{u}^{(2)}, \mathbf{u}^{(3)},\ldots \in \mathbb{S}^{d-1}$ be a sequence of directions for which there exist $0\leq t<1/2$, $\delta >0$ and $N\in \mathbb{N}$ such that
\[
|\{\mathbf{u}^{(1)},\ldots, \mathbf{u}^{(m)}\}\cap C_t(i)|\geq m(2.5 nt+\delta)
\]
for all $m\geq N$ and for all $1\leq i\leq n$. Let $P=P(\mathbf{h}_0) \in \mathcal{P}(\Delta)$ be the unknown input shape and for all $i\geq 0$ let $y^{(i)}=h_P(\mathbf{u}^{(i)})+\varepsilon ^{(i)}$ where the errors $\varepsilon ^{(i)}\sim \mathcal{N}(0,\sigma _i^2)$ are independent and normally distributed with uniformly bounded variance $\sigma _i^2\leq \gamma <\infty$. Then for all $m \geq N$ the least-squares estimate $\hat{P}^\Delta _m(U,\mathbf{y})$ is unique and
\[
\mathbb{P}\big( \lim _{m\rightarrow \infty} \hat{P}^\Delta _m(U,\mathbf{y}) = P\big)=1 \, .
\]
Moreover, for all $m\geq N$ and fixed failure probability $\eta$,
\[
\mathbb{P}\left(d_H(\hat{P}^\Delta _m(U,\mathbf{y}), P)\geq \frac{1}{\sqrt{m}}\frac{n(c^\Delta)^2}{\kappa}\sqrt {2d\gamma \lambda \log \left( \frac{2n}{\eta}\right)}\right)\leq \eta \, ,
\]
where
\[
\kappa \ = \ \left(\frac{\delta }{ \max_i \|\mathbf{v}_i\|^2}\right)^{\frac{3}{2}}
\]
and
\[
\lambda \ = \  (c^\Delta)^2-\left((c^\Delta)^2-\frac{t^2}{\min _i\|\mathbf{v}_i\|^2}\right)(n-1)(2.5nt+\delta)  \, .
\]

\end{theorem}
Roughly speaking, Theorem~\ref{thm:convergence} states that a polytope $P\in \mathcal{P}(\Delta)$ can be approximated almost surely under the assumption that the sequence of directions is sufficiently concentrated around every single ray $\mathbf{v}_i$ of $\Delta$. Note that in general we do not expect this assumption to be satisfied for uniform random unit vectors since the surface area measure of every neighborhood $C_t(i)$ can, up to constants, be bounded above and below by that of a spherical cap of radius $t$ which grows as $t^{d-1}$, that is, sublinear for small $t$. Further, we observe that the convergence rate improves with decreasing $t$. If the sequence of directions consists of facet defining vectors the assumptions on the sequence are satisfied with $t=0$ and Theorem~\ref{thm:convergence} reduces to the following special case.

\begin{corollary}\label{cor:convergence}
Let $\mathbf{u}^{(1)},\mathbf{u}^{(2)}, \mathbf{u}^{(3)},\ldots \in \mathbb{S}^{d-1}$, where $\mathbf{u}^{(i)}\in \left\{ \mathbf{v}_1/\| \mathbf{v}_1\|,\ldots, \mathbf{v}_n/\|\mathbf{v}_n\|\right\}$ for all $i\geq 1$, be a sequence of facet defining directions such that there exists a constant $\delta >0$ and an $N\in \mathbb{N}$ for which
\[
|\{\ell\in [m] \colon \mathbf{u}^{(\ell)}=\mathbf{v}_i/\| \mathbf{v}_i\|\}|\geq m\delta
\]
for all $m\geq N$ and for all $1\leq i\leq n$. Then, under the same assumptions as in Theorem~\ref{thm:convergence},
\[
\mathbb{P}_m^\Delta\big( \lim _{m\rightarrow \infty} \hat{P}(U,\mathbf{y}) = P\big)=1 \, .
\]
\end{corollary}
The remainder of this section is dedicated to proving Theorem~\ref{thm:convergence}. 

For every symmetric matrix $A\in \mathbb{R}^{n\times n}$ let $\lambda _1(A) \leq \cdots \leq \lambda _n (A)$ denote the (real) eigenvalues in increasing order. Recall that for any matrix $M\in \mathbb{R}^{m\times n}$ the symmetric matrix $M^\mathsf{T}M\in \mathbb{R}^{n\times n}$ is positive semidefinite. In particular, all eigenvalues of $M^\mathsf{T}M$ are nonnegative. The following two norm inequalities are folklore. 

\begin{lemma}\label{lem:inequalitymatrix}
Let $M\in \mathbb{R}^{m\times n}$ be a matrix. Then
\[
\sqrt{\lambda _1 (M^\mathsf{T}M)}\|\mathbf{x}\|\leq \|M\mathbf{x}\|\leq \sqrt{\lambda _n (M^\mathsf{T}M)}\|\mathbf{x}\|
\]
for all $\mathbf{x}\in \mathbb{R}^n$.
\end{lemma}
\begin{lemma}\label{lem:inequalityposdef}
Let $B\in \mathbb{R}^{n\times n}$ be a positive semidefinite matrix. Then
\[
\lambda _1 (B)\|\mathbf{x}\|\leq \|B\mathbf{x}\|\leq \lambda _n (B)\|\mathbf{x}\|
\]
for all $\mathbf{x}\in \mathbb{R}^n$.
\end{lemma}
We will use the following notations. Let $P=P(\mathbf{h}_0)$ be the unknown input shape. We denote by $\mathbf{y}_0=A_U\mathbf{h}_0\in A_U\mathcal{P}(\Delta)$ the unobservable vector of support function evaluations of $P$. In particular, $\mathbf{y}-\mathbf{y}_0=\bm{\varepsilon}$, where $\bm{\varepsilon}=(\varepsilon^{(1)},\ldots, \varepsilon^{(m)})^\mathsf{T}$ is the vector of errors. Recall that for a given $\mathbf{y}$ the unique point in the cone $A_U \mathcal{P}(\Delta)$ with minimal distance to $\mathbf{y}$ is denoted by $\hat{\mathbf{y}}$. Further, let $\bar{\mathbf{y}}$ be the orthogonal projection of $\mathbf{y}$ onto the subspace $A_U\mathbb{R}^n$.

\begin{lemma}\label{lemma:BoundsOnYY0}
With the notations above we obtain
\begin{eqnarray*}
\|\hat{\mathbf{y}}-\mathbf{y}_0\| & \leq & \|\bar{\mathbf{y}}-\mathbf{y}_0\| \, , \\
\sqrt{\lambda _1 (A_U^\mathsf{T} A_U)}\|\hat{\mathbf{h}}-\mathbf{h}_0\| & \leq & \sqrt{\lambda _n (A_U^\mathsf{T}A_U)}\|\bar{\mathbf{h}}-\mathbf{h}_0\| \, .
\end{eqnarray*}
where $\bar{\mathbf{h}}$ and $\hat{\mathbf{h}}$ denote vectors in $\mathbb{R}^n$ such that $A_U\bar{\mathbf{h}}=\bar{\mathbf{y}}$ and $A_U\hat{\mathbf{h}}=\hat{\mathbf{y}}$.
\end{lemma}
\begin{proof}
To prove the first inequality we distinguish two cases. First, if $\bar{\mathbf{y}}$ is contained in the cone $A_U\mathcal{P}(\Delta)$ then $\|\mathbf{y}-\mathbf{z}\|\geq \|\mathbf{y}-\bar{\mathbf{y}}\|$ for all $\mathbf{z}\in  A_U\mathbb{R}^n$ and thus $\bar{\mathbf{y}}=\hat{\mathbf{y}}$ since $A_U\mathcal{P}(\Delta)\subseteq A_U\mathbb{R}^n$.

If $\bar{\mathbf{y}}$ is not contained in the cone $A_U\mathcal{P}(\Delta)$ then the vectors $\bar{\mathbf{y}}-\hat{\mathbf{y}}$ and $\mathbf{y}_0-\hat{\mathbf{y}}$ enclose an obtuse angle $\varphi \in [\pi /2, \pi]$ and $\cos \varphi \leq 0$. Thus, by the law of cosine,
\[
\|\bar{\mathbf{y}}-\mathbf{y}_0\|^2 = \|\bar{\mathbf{y}}-\hat{\mathbf{y}}\|^2+\|\hat{\mathbf{y}}-\mathbf{y}_0\|^2-2\cos \varphi \|\bar{\mathbf{y}}-\hat{\mathbf{y}}\|\|\hat{\mathbf{y}}-\mathbf{y}_0\|\geq \|\hat{\mathbf{y}}-\mathbf{y}_0\|^2 
\]
which shows the first inequality. The second inequality follows from Lemma~\ref{lem:inequalitymatrix} since 
\[
\sqrt{\lambda _1 (A_U^\mathsf{T}A_U)}\|\hat{\mathbf{h}}-\mathbf{h}_0\|\leq \|\hat{\mathbf{y}}-\mathbf{y}_0\|\leq \|\bar{\mathbf{y}}-\mathbf{y}_0\|\leq \sqrt{\lambda _n(A_U^\mathsf{T}A_U)}\|\bar{\mathbf{h}}-\mathbf{h}_0\| \, . \qedhere
\]
\end{proof}

The following two lemmata provide lower and upper bounds for the eigenvalues of $A_U^\mathsf{T}A_U$.
\begin{lemma}\label{lemma:MaxEigvalAU}
For all $m\geq 1$
\[
\lambda _n (A_U^\mathsf{T} A_U)\leq mn(c^\Delta)^2 \, .
\]
\end{lemma}
\begin{proof}
For all $1\leq i\leq n$
\[
(A_U^\mathsf{T} A_U)_{ii} \ = \ \sum _{j=1}^m ([\mathbf{u}^{(j)}]_i)^2 \ \leq \ m(c^\Delta) ^2 \, .
\]
In particular,
\[
\lambda _n (A_U^\mathsf{T}A_U)\leq \sum _{i=1}^n \lambda _i (A_U^\mathsf{T} A_U) = \sum _{i=1}^n (A_U^\mathsf{T} A_U)_{ii} \leq nm(c^\Delta)^2 \, . \qedhere
\]
\end{proof}
\begin{lemma}\label{lemma:MinEigvalAU}
Let $\mathbf{u}^{(1)},\mathbf{u}^{(2)}, \mathbf{u}^{(3)},\ldots \in \mathbb{S}^{d-1}$ be a sequence of directions for which there exist $0\leq t<1/2$, $\delta >0$ and $N\in \mathbb{N}$ such that
\[
|\{\mathbf{u}^{(1)},\ldots, \mathbf{u}^{(m)}\}\cap C_t(i)|\geq m(2.5nt+\delta)
\]
for all $m\geq N$ and for all $1\leq i\leq n$. Then 
\[
\lambda _1 (A_U^\mathsf{T}A_U)\geq \frac{m\delta}{\max_{1\leq i\leq n} \|\mathbf{v}_i\|^2} \, .
\]
for all $m\geq N$.
\end{lemma}
\begin{proof}
Let $\bar{A_U}=A_U W$. Applying Lemma~\ref{lem:inequalityposdef} to the product of semidefinite matrices $A_U^\mathsf{T}A_U=W^{-1}\bar{A_U}^\mathsf{T}\bar{A_U}W^{-1}$ we obtain 
\[
\lambda _1 (A_U^\mathsf{T}A_U)\geq \lambda _1(W^{-1})^2 \lambda _1 (\bar{A_U}^\mathsf{T}\bar{A_U}) = \frac{\lambda _1 (\bar{A_U}^\mathsf{T}\bar{A_U})}{\max_{1\leq i\leq n} \|\mathbf{v}_i\|^2} \, .
\]
It thus suffices to prove $\lambda _1(\bar{A_U}^\mathsf{T}\bar{A_U})\geq m\delta$ which we will do in the following.

For any matrix $M$ let $M_i$ denote its $i$-th row. We begin by constructing two matrices $B_U,C_U \in \mathbb{R}^{m\times n}$ as follows.
\[
(B_U)_i =\begin{cases}
[\mathbf{u}^{(i)}] W& \text{ if } \mathbf{u}^{(i)} \in \bigcup _{j=1}^n C_t (j) \, ,\\
\mathbf{0}& \text{ otherwise}\, ,
\end{cases}
\]
and 
\[
(C_U)_i =\begin{cases}
[\mathbf{u}^{(i)}] W& \text{ if } \mathbf{u}^{(i)} \not \in \bigcup _{j=1}^n C_t (j) \, , \\
\mathbf{0}& \text{ otherwise} \, .
\end{cases}
\]
Then $\bar{A_U}=B_U + C_U$ and furthermore we observe that $\bar{A_U}^\mathsf{T}\bar{A_U}=B_U^\mathsf{T}B_U+C_U^\mathsf{T} C_U$. By Weyl's inequality \cite[Theorem 3.3.16]{Horn} we therefore have
\[
\lambda_1 (\bar{A_U}^\mathsf{T}\bar{A_U})\geq \lambda _1 (B_U^\mathsf{T}B_U)+\lambda _1 (C_U^\mathsf{T}C_U)\geq \lambda _1 (B_U^\mathsf{T}B_U) \, .
\]
Note that since $t<1/2$ the union $\bigcup _j C_t(j)$ is disjoint. That is, for every $i$ such that $(B_U)_i=[\mathbf{u}^{(i)}]$ is a non-zero row there exists a unique $j\in [n]$ such that $\mathbf{u}^{(i)}\in C_t(j)$. We define matrices $\tilde{B}_U$ and $D_U$ in the following way:
\[
(\tilde{B}_U)_{ij}=\begin{cases} 1 & \text{ if } \mathbf{u}^{(i)} \in C_t (j)\\
0 & \text{ otherwise.}
\end{cases}
\]
\[
(D_U)_{ij}=\begin{cases} (B_U)_{ij}-1 & \text{ if } \mathbf{u}^{(i)} \in C_t (j)\\
(B_U)_{ij} & \text{ otherwise.}
\end{cases}
\]
Then, by definition, $B_U=\tilde{B}_U+D_U$ and furthermore $B_U^\mathsf{T}B_U=\tilde{B}_U^\mathsf{T}\tilde{B}_U+\tilde{D}$ where
\[
\tilde{D}=D_U^\mathsf{T}D_U+D_U^\mathsf{T}\tilde{B}_U+\tilde{B}_U^\mathsf{T}D_U \, .
\]
By the Hoffman-Wieland inequality \cite{HoffmanWielandt}
\[
(\lambda _1 (B_U^\mathsf{T}B_U)-\lambda _1 (\tilde{B}_U^\mathsf{T}\tilde{B}_U))^2 \leq \sum _{i=1}^n\sum _{j=1}^n \tilde{D}_{ij}^2 \, .
\]
We observe that by construction every element in $D_U$ is in the interval $[-t,t]$. We therefore have
\[
| (D_U^\mathsf{T} D_U)_{ij}|\leq mt^2
\]
for all $1\leq i,j\leq n$ and similarly
\[
|(D_U^\mathsf{T}\tilde{B}_U )_{ij}|\leq mt \quad \text{ and } |(\tilde{B}_U^\mathsf{T}D_U)_{ij}| \leq mt \, .
\]
In particular, since $0\leq t< 1/2$
\[
|\tilde{D}_{ij}|\leq (t+2)mt \leq 2.5mt
\]
 and thus we obtain
 \[
 (\lambda _1 (B_U^\mathsf{T}B_U)-\lambda _1 (\tilde{B}_U^\mathsf{T}\tilde{B}))^2 \leq (2.5)^2 m^2t^2n^2
 \]
 from which 
 \begin{equation}\label{eq:conv1}
  \lambda _1 (B_U^\mathsf{T}B_U)\geq \lambda _1(\tilde{B}_U^\mathsf{T}\tilde{B}) -2.5 mtn
 \end{equation}
 follows. We observe that by construction $\tilde{B}_U^\mathsf{T}\tilde{B}_U$ is a diagonal matrix with
 \[
 (\tilde{B}_U^\mathsf{T}\tilde{B}_U)_{ii} = |\{\mathbf{u}^{(1)},\ldots, \mathbf{u}^{(m)}\} \cap C_t (i)|
 \]
 for all $1\leq i\leq n$. By assumption and Equation~\eqref{eq:conv1} we therefore obtain
 \[
 \lambda _1 (\bar{A_U}^\mathsf{T} \bar{A_U})\geq \lambda _1 (B_U^\mathsf{T}B_U)\geq m(2.5nt+\delta)-2.5mtn=m\delta
 \]
 which concludes the proof.

\end{proof}

\begin{proof}[{Proof of Theorem~\ref{thm:convergence}}] First we observe that under the assumptions of Theorem~\ref{thm:convergence} it follows from Lemma~\ref{lemma:MinEigvalAU} that $A_U^\mathsf{T}A_U$ is invertible for all $m\geq N$. In particular, $A_U$ has rank $n$ in this case and the least-squares estimator $\hat{P}^\Delta _m (U,\mathbf{y})$ is unique for all $\mathbf{y}\in \relint A_U \mathcal{P}(\Delta)$ by Proposition~\ref{prop:uniquenessy}.
By definition, $\bm{\varepsilon} =\mathbf{y}-A_U\mathbf{h}_0$. Since $\bar{\mathbf{y}}-\mathbf{y}$ is in the orthogonal complement of $A_U\mathbb{R}^n$ we have
\[
A_U^\mathsf{T} \bm{\varepsilon }= A_U^\mathsf{T} (\mathbf{y}-\bar{\mathbf{y}})+A_U^\mathsf{T}(\bar{\mathbf{y}}-A_U\mathbf{h}_0)=A_U^\mathsf{T}(\bar{\mathbf{y}}-A_U\mathbf{h}_0)
\]
and thus, by Lemmata \ref{lemma:BoundsOnYY0}, \ref{lemma:MaxEigvalAU}, \ref{lemma:MinEigvalAU} for all $m\geq N$
\begin{eqnarray*}
\|A_U^\mathsf{T} \bm{\varepsilon} \| & \geq & \lambda _1 (A_U^\mathsf{T}A_U)\|\bar{\mathbf{h}}-\mathbf{h}_0\|\\
&\geq & \frac{\left(\sqrt{\lambda _1 (A_U^\mathsf{T}A_U)}\right)^3}{\sqrt{\lambda _n (A_U^\mathsf{T}A_U)}}\|\hat{\mathbf{h}}-\mathbf{h}_0\|\\
&\geq & m \frac{\delta^{\frac{3}{2}}}{\sqrt{n}c^\Delta (\max _i \|\mathbf{v}_i\|^2)^{\frac{3}{2}}}\|\hat{\mathbf{h}}-\mathbf{h}_0\| \, .
\end{eqnarray*} 
Since all entries of $A_U$ are bounded by $c^{\Delta}$ and the error terms $\{\bm\varepsilon ^{(i)}\}_{i\geq 1}$ are pairwise independent and of finite, uniformly bounded variance, by the law of large numbers we have that
\[
\frac{1}{m}A_U^\mathsf{T} \bm\varepsilon \rightarrow \mathbf{0}
\]
almost surely. Thus also $\hat{\mathbf{h}}$ converges to $\mathbf{h}_0$ almost surely. Therefore $P(\hat{\mathbf{h}})=\hat{P}^\Delta _m(U,\mathbf{y})$ converges to $P(\mathbf{h}_0)$ in Hausdorff distance almost surely by Lemma~\ref{lem:Hausdorffdist} as claimed. 

Further, by using the inequality above and Lemma~\ref{lem:Hausdorffdist}, for all $s\geq 0$ we obtain
\begin{eqnarray}
\mathbb{P}(d_H(P(\hat{\mathbf{h}}),P(\mathbf{h}_0))\geq s)&\leq &\mathbb{P}\left(\|\hat{\mathbf{h}}-\mathbf{h}_0\|\geq \frac{s}{\sqrt{d}c^\Delta}\right)\\
&\leq &\mathbb{P}\left(\|A_U^T\bm{\varepsilon}\|\geq \frac{sm\kappa}{\sqrt{dn}(c^\Delta)^2}\right)\\
&\leq &\mathbb{P}\left(\|A_U^T\bm{\varepsilon}\|_\infty\geq \frac{sm\kappa}{\sqrt{d}(c^\Delta)^2 n}\right)\\
&\leq &\sum _{i=1}^n\mathbb{P}\left(|(A_U^T\bm{\varepsilon})_i|\geq \frac{sm\kappa}{\sqrt{d}(c^\Delta)^2 n}\right)\label{eq:rate1}
\end{eqnarray}
For all $1\leq i\leq n$,
\[
(A_U^T\bm{\varepsilon})_i \ = \ \sum _{j=1}^n [\mathbf{u}^{(j)}]_i \varepsilon^{(j)}
\]
is normally distributed with mean zero and variance $\sigma ^2 = \sum _{j=1}^n ([\mathbf{u}^{(j)}]_i)^2 \sigma _j^2$. By definition of $c^\Delta$, we have $[\mathbf{u}^{(j)}]_i\leq c^\Delta$ for all $i,j$. Moreover, by definition of $C_t(k)$, for all $k\neq i$ and all $\mathbf{u}^{(j)}\in C_t(k)$ we have $[\mathbf{u}^{(j)}]_i\leq \tfrac{t}{\|\mathbf{v}_i\|}$. Since $t<1/2$ the union $\bigcup _ k C_t(k)$ is disjoint and by assumption we have for all $m\geq N$
\begin{eqnarray*}
\sigma ^2 &=&\sum _{j=1}^n ([\mathbf{u}^{(j)}]_i)^2 \sigma _j^2\\
&\leq&\sum _{k\neq i} \sum _{j\colon \mathbf{u}^{(j)}\in C_t(k)}\frac{t^2}{\|\mathbf{v}_i\|^2}\gamma + \sum _{j:\mathbf{u}^{(j)}\not \in \bigcup _ {k\neq i} C_t(k)}(c^{\Delta})^2 \gamma\\
&\leq &\gamma \cdot \left((n-1)m(2.5 nt+\delta)\frac{t^2}{\|\mathbf{v}_i\|^2}+(m-(n-1)m(2.5 nt+\delta))(c^{\Delta})^2 \right)\\
&=&m\gamma \left((c^\Delta)^2-\left((c^\Delta)^2-\frac{t^2}{\|\mathbf{v}_i\|^2}\right)(n-1)(2.5nt+\delta) \right)\\
&\leq & m\gamma \lambda
\end{eqnarray*}
Applying standard tail bounds for Gaussian random variables to~\eqref{eq:rate1} yields
\[
\mathbb{P}(d_H(P(\hat{\mathbf{h}}),P(\mathbf{h}_0))\geq s) \leq 2n \exp \left(-\frac{s^2m\kappa^2}{2d\gamma \lambda n^2(c^\Delta)^4}\right) \, .
\]
Thus, setting 
\[
s \ = \ \frac{1}{\sqrt{m}}\frac{n(c^\Delta)^2}{\kappa}\sqrt {2d\gamma \lambda \log \left( \frac{2n}{\eta}\right)}
\]
yields the claim.
\end{proof}

\section{Limitations}\label{sec:ReconstrFixedFacetNormals}
In the previous sections we focus on the reconstruction of polytopes with fixed facet directions and fixed normal fan, that is, fixed combinatorial type. In this section we demonstrate the limitations of our results if the assumption of fixed normal fan is removed.
\subsection{Uniqueness}\label{subsec:arbitraryuniqueness}
In dimension $2$ the facet directions of a polytope determine the combinatorial type of the polytope, up to deformation. In dimension $3$ or higher this is no longer the case, as we illustrated in Example~\ref{ex:3d}.

In Theorem~\ref{thm:main} we prove that for fixed combinatorial type the least-squares estimator is given by a quadratic program. Using this result in Proposition~\ref{prop:uniquenessy} we conclude that the reconstruction is unique whenever the matrix $A_U$ has full rank. In Corollary~\ref{cor:piecewisequadratic} we see that if we remove the restriction on the combinatorial type the solution set is given by a piecewise quadratic program. We now give an example that shows that in general for given directions this piecewise quadratic program does not have a unique global minimum even though $A^\Delta_U$ has full rank for all possible deformation cones $\Delta$.  Moreover, the solution set can even be disconnected. 

Let $V$, $P_1=P(\mathbf{h}_1)$ and $P_2=P(\mathbf{h}_2)$ where $\mathbf{h}_1=(4,4,2,2,0)^\mathsf{T}$ and $\mathbf{h}_2=(2,2,4,4,0)^\mathsf{T}$
be defined as in Example~\ref{ex:3d}. 

We consider the following data set of directions. Let
\[
U =\left\{\begin{pmatrix} 1 \\ 1 \\-1\end{pmatrix}, \begin{pmatrix} 1 \\ -1 \\0\end{pmatrix}, \begin{pmatrix} -1 \\ 1 \\0\end{pmatrix}, \begin{pmatrix} -1 \\ -1 \\0\end{pmatrix}, \begin{pmatrix} 1 \\ 1 \\6\end{pmatrix}, \begin{pmatrix} -1 \\ -1 \\4\end{pmatrix}\right\} \, .
\] 
With the notation as in Example~\ref{ex:3d} we observe that $\mathbf{u}^{(i)} \in \beta_i = \gamma_i$ for $i \in [4]$, and $\mathbf{u}^{(j)} \in \beta_j$, and $\mathbf{u}^{(j)} \in \gamma_j$ for $j \in \{5,6\}$. From that we can determine the matrices
\[
A_U^{\Delta _1} = \begin{pmatrix}
1 & 0 & 1 & 0 & 3  \\
 0 & 1 & 1 & 0 & 2 \\
 1 & 0 & 0 & 1 & 2 \\
 0 & 1 & 0 & 1& 2 \\
 1 & 0 & 3 & 2  & 0 \\
 0 & 1 & 1 & 2 & 0 
\end{pmatrix}, \quad \quad A_U^{\Delta _2} = \begin{pmatrix}
1 & 0 & 1 & 0 & 3  \\
 0 & 1 & 1 & 0 & 2 \\
 1 & 0 & 0 & 1 & 2 \\
 0 & 1 & 0 & 1& 2 \\
 3 & 2 & 1 & 0  & 0 \\
 1 & 2 & 0 & 1 & 0 
\end{pmatrix}.
\]
The vector of support function evaluations $(\mathbf{h}_{P_j}(\mathbf{u}^{(i)}))_{i=1}^6$ is equal to $A_U^{\Delta _j}\mathbf{h}_j=(6,6,6,6,10,10)^\mathsf{T}$ for both $j=1$ and $j=2$. In particular, for $\mathbf{y}=(6,6,6,6,10,10)^\mathsf{T}$ and $j=1,2$
\[
\frac{1}{6}\sum _{i=1}^6(\mathbf{h}_{P_j}(\mathbf{u}^{(i)})-y^{(i)})^2=0 \, .
\]
Since the matrices $A_U^{\Delta _1}$ and $A_U^{\Delta _2}$ have rank $5$, by Proposition~\ref{prop:uniquenessy}, $P_1$ and $P_2$ are the unique reconstructions within the classes $\mathcal{P}(\Delta _1)$ and $\mathcal{P}(\Delta _2)$, respectively. However, within the bigger class of polytopes $\mathcal{K}$ with facet directions restricted to $V$ both polytopes $P_1$ and $P_2$ are least-squares estimators and thus the reconstruction is no longer unique. 

\subsection{Convergence}\label{subsec:arbitraryconvergence}
In Section~\ref{sec:convergence} we consider convergence properties of our Algorithm~\ref{AlgoComplexity} for a sequence of directions $\mathbf{u}^{(1)},\mathbf{u}^{(2)},\ldots$ under the assumption that the unknown underlying polytope has a particular normal fan $\Delta$. In this section we will see that the assumption on the normal fan is crucial for the convergence. To see that, in the following we give an example in which all assumptions of Corollary~\ref{cor:convergence} are satisfied, apart from the condition that $P(\mathbf{h}_0)$ is contained in $\mathcal{P}(\Delta)$.

We again consider the polytopes $P_1$ and $P_2$ defined in Example~\ref{ex:3d}. With the same notation as in the example, we consider the following two sequences of directions and support function evaluations:
For any $i \geq 1$ let
\[ \mathbf{u}^{(i)} =
  \begin{cases}
    \mathbf{v}_1       & \quad \text{if } i = 0, 1,2 \mod 10 \\
    \mathbf{v}_2  & \quad \text{if } i = 3, 4, 5  \mod 10\\
    \mathbf{v}_3  & \quad \text{if } i = 6 \mod 10\\
    \mathbf{v}_4  & \quad \text{if } i = 7 \mod 10\\
    \mathbf{v}_5 & \quad \text{if } i=8,9 \mod 10
  \end{cases}
\]
and $y^{(i)}=h_{P_2}(\mathbf{u}^{(i)})$, and 
 \[ \tilde{\mathbf{u}}^{(i)} =
  \begin{cases}
    \mathbf{v}_1       & \quad \text{if } i = 0 \mod 10 \\
    \mathbf{v}_2  & \quad \text{if } i = 1, \ldots 6 \mod 10\\
    \mathbf{v}_3  & \quad \text{if } i = 7  \mod 10\\
    \mathbf{v}_4  & \quad \text{if } i = 8 \mod 10\\
      \mathbf{v}_5 & \quad \text{if } i=9 \mod 10
  \end{cases}
  \]
and $\tilde{y}^{(i)}=h_{P_2}(\tilde{\mathbf{u}}^{(i)})$.\\

We restrict ourselves to reconstructions within the set of polytopes in $\mathcal{P}(\Delta _1)$ where $\Delta _1$ is the normal fan of $P_1$. Since every polytope with facet directions $V$ is uniquely determined by the support function evaluations in directions $\mathbf{v}_1,\ldots, \mathbf{v}_n$, the true underlying polytope for both data sets defined above, namely $P_2$, does not lie in $\mathcal{P}(\Delta _1)$. Thus, the least-squares estimator cannot converge to the true underlying body. In fact, it may not converge at all: Both sequences of data satisfy the assumptions of Corollary~\ref{cor:convergence}. For all $k\geq 1$ and $m=10k$, we obtain $\hat{\mathbf{h}}$ as the solution of $\argmin _{P(\mathbf{h})\in \mathcal{P}(\Delta_1)}\| A_j\mathbf{h}-\mathbf{y}\|$, for $j=1,2$, where
\[
A_1 =\begin{bmatrix}
B_1^\mathsf{T} \\
\vdots\\
B_1^\mathsf{T}
\end{bmatrix} \in \mathbb{R}^{m \times 5}, ~~~~~~
B_1 =\begin{bmatrix}
1 & 1 & 1 & 0 & 0& 0  & 0 & 0 & 0 & 0\\
0 & 0 & 0 & 1 & 1& 1  & 0 & 0 & 0 & 0\\
0 & 0 & 0 & 0 & 0& 0  & 1 & 0 & 0 & 0\\
0 & 0 & 0 & 0 & 0& 0  & 0 & 1 & 0 & 0\\
0 & 0 & 0 & 0 & 0& 0  & 0 & 0 & 1 & 1
\end{bmatrix}
\]
for the first data sequence, and 
\[
A_2 =\begin{bmatrix}
B_2^\mathsf{T} \\
\vdots\\
B_2^\mathsf{T}
\end{bmatrix} \in \mathbb{R}^{m \times 5}, ~~~~~~
B_2 =\begin{bmatrix}
1 & 0 & 0 & 0 & 0& 0  & 1 & 0 & 0 & 0\\
0 & 1 & 1 & 1 & 1 & 1& 1  & 0 & 0 & 0 \\
0 & 0 & 0 & 0 & 0& 0  & 0 & 1 & 0 & 0\\
0 & 0 & 0 & 0 & 0& 0  & 0 & 0 & 1 & 0\\
0 & 0 & 0 & 0 & 0& 0  & 0 & 0 & 0 & 1
\end{bmatrix}
\]
for the second data sequence. In the first case, $\hat{\mathbf{h}}=(2.5,2.5,2.5,2.5,0)^\mathsf{T}$ while in the second case $\hat{\mathbf{h}}=(62/19,42/19, 52/19,52/19,0)^\mathsf{T} \approx (3.26,2.21,2.73,2.73,0)^\mathsf{T}$. See Figure~\ref{fig:convergence}. Note that $\| A_j\mathbf{h}-\mathbf{y}\|$ equals $0$ for $\mathbf{h} = \mathbf{h}_2 = (2,2,4,4,0)^\mathsf{T}$ but this vector is not an element of $\mathcal{P}(\Delta _1)$. For the first sequence of data the first two elements of $\mathbf{h}$ have the most effect on the value for $\| A_j\mathbf{h}-\mathbf{y}\|$. Furthermore, because of the symmetries and the symmetric distribution of data points, the optimal solution is symmetric in the first $4$ elements of $\mathbf{h}$. In this case, $P(\mathbf{h})$ has just $5$ vertices, since $w_5$ and $w_6$ coincide. In the second sequence of data, the main contribution to $\| A_j\mathbf{h}-\mathbf{y}\|$ is given by the second element of $\mathbf{h}$. Therefore the facet with normal vector $\mathbf{v}_2$  in $P_2$  coincides almost with the corresponding facet of $P((62/19,42/19, 52/19,52/19,0)^\mathsf{T})$. In this case $w_5$ and $w_6$ are close but do not coincide.

\begin{figure}\label{fig:convergence}
\centering
\includegraphics[scale=0.132]{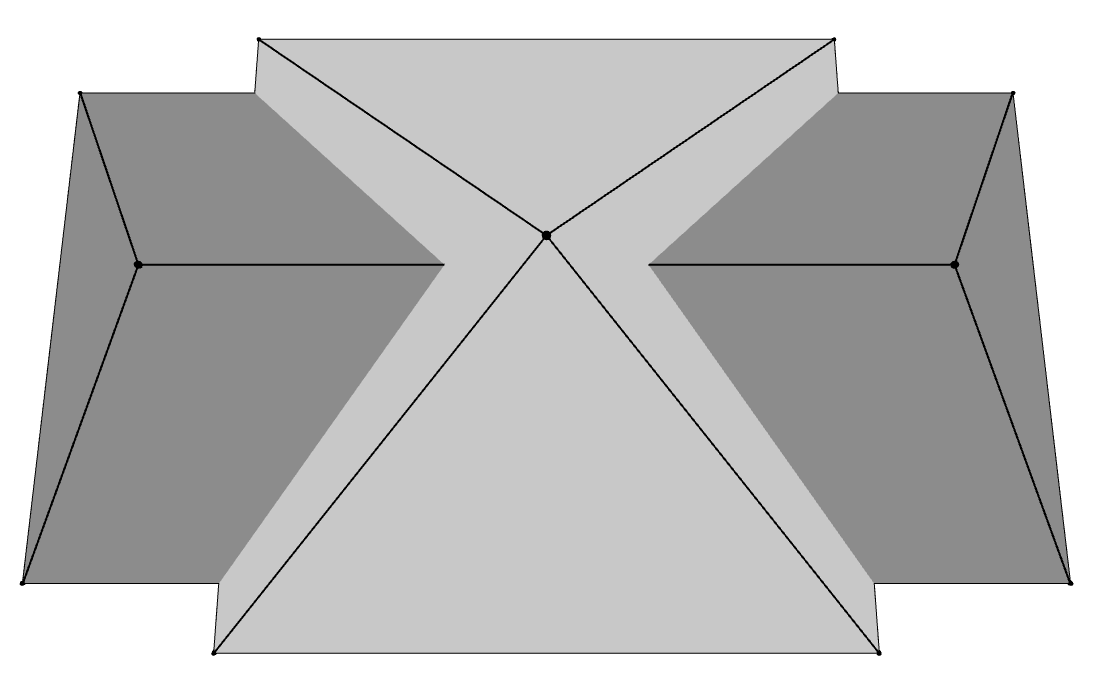}
\includegraphics[scale=0.14]{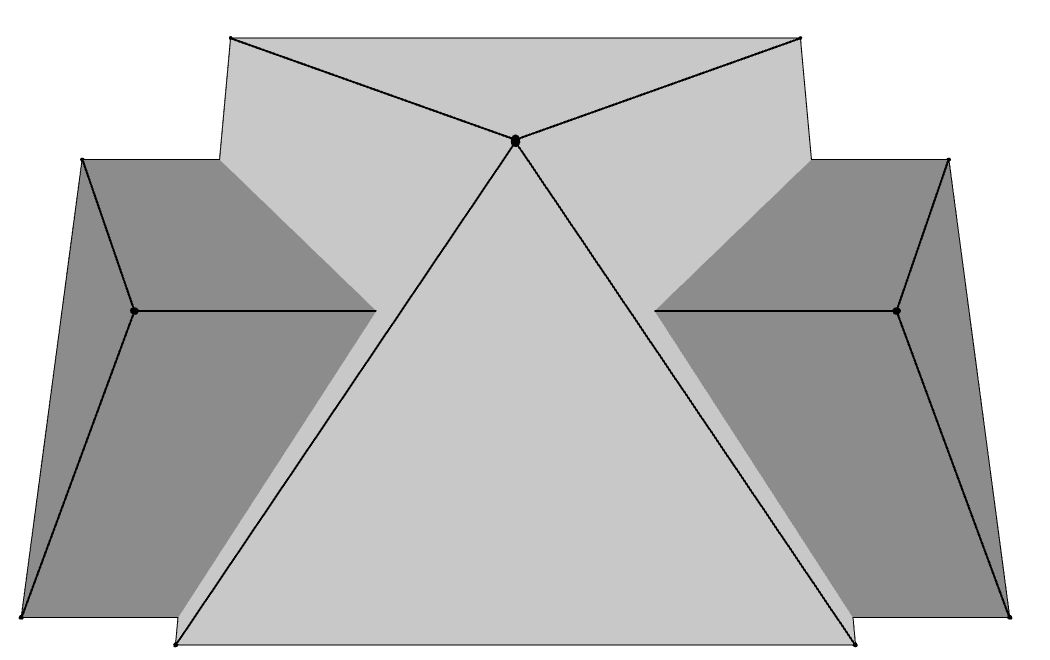}
\caption{Left: Reconstruction for the first data set with optimal solution $P((2.5,2.5,2.5,2.5,0)^\mathsf{T})$. Right: Optimal solution $P((3.26,2.21,2.73,2.73,0)^\mathsf{T})$ for the second data set. }
\end{figure}

This shows that the assumptions of Corollary~\ref{cor:convergence} are not sufficient for convergence if the unknown underlying polytope does not have the given normal fan.\\

\section{Concluding remarks}
In this article, we investigated the reconstruction of polytopes with fixed facet directions from support function evaluations. Our emphasis was on polytopes with given normal fan and their deformations for which we studied questions about the geometry of the solution set, uniqueness of the reconstruction as well as convergence. We also discussed limitations of our results if the restriction on the deformation cone is removed. This naturally raises the question about how one can select a suitable deformation cone for the reconstruction. This is highly relevant as approximating a polytope with a class of polytopes in a different deformation cone may lead to poor outcome. For instance, in Example~\ref{ex:3d}, given a polytope in $\mathcal{P}(\Delta _1)$, the best approximation with a polytope in $\mathcal{P}(\Delta _2)$ will be a pyramid with a square base whose Hausdorff distance to the original polytope can be arbitrarily large. In general, the number of deformation cones can be huge, even in small dimensions. Every deformation cone of a simple polytope with facet directions $\mathbf{v}_1,\ldots, \mathbf{v}_n$ corresponds to a regular central triangulation of the vector configuration $\mathbf{v}_1,\ldots, \mathbf{v}_n$ (see~\cite[Theorem 9.5.6.]{Triangulations}). For example, in~\cite{joswig2019parametric}, the number of triangulations was studied in case if the normal vectors are given by $e_i-e_j$, $1\leq i,j\leq n$, $i\neq j$; for $n=5$, there are $27248$ combinatorial types, up to symmetry~\cite[Theorem 24]{joswig2019parametric}. Developing an algorithm that selects an appropriate deformation cone for given data and possibly prior information about the underlying polytope remains an interesting and probably very challenging open question.\\

\textbf{Acknowledgement:} We would like to thank the annonymous referees for carefully reading our manuscript and many insightful comments that helped us improving the paper. We also would like to thank Diane Holcomb for helpful discussions concerning Section~\ref{sec:convergence}. MD and KJ were partially supported by the Wallenberg AI, Autonomous Systems and Software program. KJ was furthermore partially supported by grant 2018-03968 of the Swedish research council as well as the G\"oran Gustafsson Foundation.

\bibliographystyle{siam}
\bibliography{references}

\end{document}